\documentclass[11pt,leqno]{amsart}
\usepackage[margin=3.8cm]{geometry}
\usepackage[utf8]{inputenc}
\usepackage{graphicx, array, url, tcolorbox, xcolor, hyperref, amsmath, amsfonts, amsthm, amssymb, amstext}
\usepackage{tikz}
\usepackage{float}
\usepackage{caption}
\usetikzlibrary{arrows.meta}
\usepackage{caption}
\usepackage{todonotes}

\newcommand{\N}{\mathbb N}
\newcommand{\C}{\mathbb C}
\newcommand{\D}{\mathbb D}
\newcommand{\hatc}{\hat{\mathbb C}}
\newcommand{\R}{\mathbb R}
\newcommand{\Z}{\mathbb Z}
\newcommand{\Sph}{\mathbb S}
\newcommand{\eps}{\epsilon}

\newcommand{\bi}{{\mathbf i}}
\newcommand{\cI}{{\mathcal I}}
\DeclareMathOperator{\diam}{diam}
\DeclareMathOperator{\dist}{dist}

\DeclareMathOperator{\ndim}{\dim_N}

\DeclareMathOperator{\id}{id}

\definecolor{peach}{HTML}{F7965A}
\definecolor{springgreen}{HTML}{C6DC67}
\definecolor{royalblue}{HTML}{0071BC}
\definecolor{periwinkle}{HTML}{7977B8}
\definecolor{pinegreen}{HTML}{008B72}
\definecolor{olivegreen}{HTML}{3C8031}

\numberwithin{equation}{section}

\newcounter{sec_counter} \numberwithin{sec_counter}{section}

\newtheorem{theorem}[sec_counter]{Theorem}
\newtheorem{lemma}[sec_counter]{Lemma}
\newtheorem{corollary}[sec_counter]{Corollary}
\newtheorem{proposition}[sec_counter]{Proposition}

\theoremstyle{definition}
\newtheorem{definition}[sec_counter]{Definition}

\theoremstyle{remark}
\newtheorem{remark}[sec_counter]{Remark}

\newtheorem{example}[sec_counter]{Example}

\newtheorem{question}[sec_counter]{Question}

\title[Analytic and quasiregular distortion of Nagata dimension]{Analytic and quasiregular \\ distortion of Nagata dimension}

\author{Manisha Garg}
\address{Department of Mathematics, University of Illinois at Urbana-Champaign, Urbana, Illinois, USA}
\email{manisha8@illinois.edu}

\author{Jeremy T. Tyson}
\address{Department of Mathematics, University of Illinois at Urbana-Champaign, Urbana, Illinois, USA}
\email{tyson@illinois.edu}

\date{\today}


\begin{document}

\maketitle

\begin{abstract}
We study how analytic functions, and more generally quasiregular mappings, distort Nagata dimension. Quasiconformal mappings of domains preserve the Nagata dimension of compact subsets, in view of a result of Lang and Schlichenmaier. We establish the same conclusion for analytic functions defined on general planar domains. On the other hand, polynomials (and more generally, rational maps) preserve the Nagata dimension of arbitrary subsets of their domain. In the absence of the compactness assumption, we provide examples to show that an entire function can increase or decrease the Nagata dimension of subsets of the domain. Some of these results generalize to meromorphic functions, and separately to planar quasiregular maps in view of Sto{\"\i}low factorization. We also show that conformal mappings can change the porosity behavior of non-compact subsets of their domain; this yields examples of planar conformal maps which take sets of Nagata dimension strictly less than two onto set of Nagata dimension two.
We conclude with open questions and potential future work related to the distortion of Nagata dimension by higher-dimensional quasiregular maps.
\end{abstract}

\tableofcontents

\section{Introduction}

The Nagata dimension (also known as Assouad--Nagata dimension) of a metric space, denoted $\dim_N X$, is a geometric invariant that parallels the notions of topological and asymptotic dimensions. While the topological dimension captures the local, small-scale structure of a space and is preserved under homeomorphisms, and the asymptotic dimension is a large-scale invariant stable under quasi-isometries, the Nagata dimension is simultaneously sensitive to both fine and coarse geometric features of a metric space. 

Because the Nagata dimension provides quantitative control of coverings at all scales, it plays a significant role in the large scale geometry of groups. In particular, Nagata dimension has been used by Buyalo-Lebedeva \cite{BuyaloLebedevaPontryagin2007} to quasi-isometrically classify Gromov hyperbolic groups via the geometry of their boundaries. A closely related invariant, originally introduced under the name capacity dimension \cite{BuyaloCap22005} and later referred to as linearly controlled dimension (see, for instance, \cite{BuyaloSchroeder:asympbook}), corresponds to the small-scale version of the Nagata dimension. Both notions coincide for bounded metric spaces, making it a significant tool for geometric group theory and coarse geometry (see e.g.\ \cite{BuyaloCap12005}, \cite{BrodskiyDydakLevinMitra2008}, or \cite{DydakHiges2008}). 

From the perspective of mapping theory, studying dimensions and their invariance or distortion under various classes of maps provides a means of identifying which metric spaces embed `nicely' into Euclidean spaces.
For instance, Assouad’s embedding theorem \cite{ass:plongements} states that, for any metric space $(X,d)$ with finite Assouad dimension $\dim_A X$ and each $\varepsilon \in (0,1)$, there exists an integer $n$ such that the snowflaked metric space $(X, d^{\varepsilon})$ admits a bi-Lipschitz embedding into $\R^n$. A similar embedding assertion holds true for the Nagata dimension $\ndim X$, to wit, if $(X,d)$ is a metric space of finite Nagata dimension $n$, then for every $\varepsilon \in (0,1)$, the space $(X,d^\varepsilon)$ admits a Lipschitz light mapping into $\R^n$. See \cite{DavidNagataEmbed2024} for details. In fact, the Assouad and Nagata dimensions are closely related to each other. It is shown in \cite{DonneRajala2015} that $\ndim X \le \dim_A X$ for all metric spaces $(X,d)$, and both invariants capture covering properties of the space at all scales. 

Beyond its role in embedding theorems, the Assouad dimension has also been studied from the perspective of mapping distortion. For instance, the distortion of Assouad dimension under quasiconformal or quasiregular maps has been studied in \cite{tys:assouad}, \cite{ChrontsiosTyson2023} and \cite{Chrontsios2024}. Quasiconformal mappings are locally quasisymmetric, and it is well known that quasisymmetric deformations, such as snowflaking, can alter the Assouad dimension of a space. In contrast, Lang and Schlichenmaier \cite{LangSchlichenmaier2005} proved that the Nagata dimension is invariant under quasisymmetric maps, and Xie extended such invariance to the broader class of quasi-M\"obius maps \cite{Xie2008}. 

It is natural to ask whether quasiconformal and quasiregular mappings preserve Nagata dimension, and if not, to what extent this notion of dimension can be distorted. Since quasiconformal mappings are locally quasisymmetric, it follows directly from \cite{LangSchlichenmaier2005} that such maps preserve the Nagata dimension of compact subsets of their domains. However, such invariance fails in full generality: the exponential map $f(z) = e^z$, restricted to a suitable domain on which it is univalent, maps an arithmetic sequence (of Nagata dimension one) to a geometric sequence (of Nagata dimension zero).

According to Sto{\"\i}low’s factorization theorem, every planar quasiregular mapping is the composition of an analytic function and a quasiconformal map. Hence, the behavior of the Nagata dimension under planar quasiregular mappings is directly related to its behavior under analytic functions.

In this paper we investigate how analytic and planar quasiconformal mappings affect the Nagata dimension of subsets of their domain. We show that analytic functions exhibit a range of possible behaviors: they can preserve, increase, or decrease Nagata dimension. We first prove that non-constant polynomials preserve the Nagata dimension of any set. In fact, polynomials are the only entire mappings which preserve Nagata dimension of any subset of their domain (Corollary \ref{cor:polynomial-characterization}).

\begin{theorem}\label{mainthm1:polynomial_preserve_ndim}
Let $f$ be a non-constant polynomial. Then $\ndim X = \ndim f(X)$ for every $X \subset \C$.
\end{theorem}

On the other hand, entire functions with an essential singularity at infinity exhibit strikingly different behavior with respect to the Nagata dimension of subsets of their domains.

\begin{theorem}[Nagata dimension increase by entire functions]\label{thm:entire_func_increase_ndim}
Let $f$ be a non-constant entire function with an essential singularity at infinity. For any $Y \subset \C$, there exists a countable set $X$ with Nagata dimension zero so that $f(X)$ is dense in $Y$ (and hence $\ndim f(X)  \;=\; \ndim Y $).
\end{theorem}

Theorems \ref{mainthm1:polynomial_preserve_ndim} and \ref{thm:entire_func_increase_ndim} immediately yield the following corollary.

\begin{corollary}\label{cor:polynomial-characterization}
Let $f:\C\to\C$ be entire. The following are equivalent:
\begin{enumerate}
\item $f$ is a non-constant polynomial;
\item $\dim_N f(X) =\dim_N X$ for every $X \subset \C$.
\end{enumerate}
\end{corollary}

Analogous results also hold true for rational functions on the Riemann sphere, equipped with the spherical metric $q$.

\begin{theorem}\label{mainthm2:rational_preserve_ndim}
Let $r: (\Hat{\C},q) \to (\Hat{\C}, q)$ be a non-constant rational function. Then $\ndim X  = \ndim r(X) $ for every $X\subset \Hat{\C}$. Moreover, a nonconstant meromorphic function $f:\C \to \Hat{\C}$ is rational if and only if it preserves the Nagata dimension of every subset of $\C$.
\end{theorem}

Entire functions can also decrease the Nagata dimension of subsets of their domain. For instance, the exponential map $f(z)=e^z$ maps the half-lattice
\[
A=\{ k + 2\pi \bi \ell : k\in\N,\ \ell\in\Z \}
\]
onto the geometric sequence $\{e^k:k\in\N\}$, and $\dim_N A = 2$ while $\dim_N f(A)=0$. Quasiregular mappings in higher dimensions may exhibit similar behavior; see, for example, the discussion of Zorich's mapping in Example~\ref{ex:zorich}. The following theorem provides another example of Nagata dimension decrease.

\begin{theorem}[Nagata dimension decrease under entire functions]\label{thm:entire_genus_1_decrease_ndim}
Let $A=\{\alpha_j\}_{j\in\N}$ be a discrete set in $\C$ with
\begin{equation}\label{eq:genus-one}
\sum_{j}|\alpha_j|^{-1}<\infty,
\end{equation}
let $q$ be a non-constant polynomial, and let $m \in \N \{0\}$. Define
\[
        F(z)\;:= \;z^m e^{q(z)}\prod_{j=1}^{\infty}\Bigl(1-\frac{z}{\alpha_j}\Bigr).
\]
Then there exists a set $X\subset\C$ such that $\ndim X = 1$ and $ \ndim F(X) = 0$.
\end{theorem}

The proof of Theorem \ref{thm:entire_genus_1_decrease_ndim} is substantially more difficult and technical than that of Theorem \ref{thm:entire_func_increase_ndim}. The restriction in Theorem \ref{thm:entire_genus_1_decrease_ndim} to entire functions whose zeroset $A$ satisfies the genus one condition \eqref{eq:genus-one} is not essential; it streamlines certain aspects of the proof. We expect that a similar conclusion holds true for entire functions of arbitrary genus.

The distortion of dimension under quasiconformal mappings has been studied extensively since the foundational work of Gehring and V\"ais\"al\"a, \cite{gv:hausdorff}. In the planar case, they showed that quasiconformal maps distort Hausdorff dimension in a controlled manner, and in particular, preserve the dimension of sets of dimension two. As noted above, quasiconformal maps preserve Nagata dimension of compact subsets of their domain.

Porosity of sets has been investigated in connection with dynamical systems, mapping theory, and fractal geometry. Informally, a set $E \subset \R^n$ is porous in $\R^n$ if every ball $B(x,r)$ in $\R^n$ contains a subball of comparable radius contained fully in $\R^n \setminus E$. For the precise definition, see \S~\ref{sec:porosity}. Germane to this paper is the relationship between porosity and Nagata dimension. To wit, a set $E \subset \R^n$ is porous in $\R^n$ if and only if $\ndim E \le n-1$. See, e.g., Theorem \ref{th:equivalence}.

As observed by V\"ais\"al\"a in \cite[Introduction]{Vaisala1987Porous}, if $f:\R^n \to \R^n$ is quasiconformal and $E$ is porous in $\R^n$, then $f(E)$ is also porous in $\R^n$. In \cite[Theorem 4.2]{Vaisala1987Porous} V\"ais\"al\"a establishes the (more difficult) fact that the same conclusion holds true if $f$ is only assumed to be a quasisymmetric embedding of $E \subset \R^n$ into $\R^n$. 

It follows from the previous remarks that quasiconformal maps $f:\Omega \to \Omega'$ between domains in $\R^n$ preserve porosity (in $\R^n$) for compact subsets of $\Omega$. Such a conclusion fails to hold, however, if the compactness restriction is removed. In fact, planar conformal maps can map sets of Nagata dimension strictly less than two to sets of Nagata dimension two. Denote by $\D$ the open unit disc in $\C$.

\begin{theorem}[Conformal mappings can destroy porosity]\label{thm:spiral_domain_porosity_change}
There exists a simply connected domain $\Omega \subsetneq \C$,  a conformal map $f : \D \to \Omega$, and a diameter $\ell \subset \D$ such that $\ndim \ell = 1$ and $\ndim f(\ell) = 2$. In particular, $f$ maps a porous subset $\ell$ of $\C$ to a non-porous subset $f(\ell) \subset \C$.
\end{theorem}

Since porosity of subsets of $\R^n$ is also characterized by the condition that the Assouad dimension is strictly less than $n$ \cite{luu:assouad}, we also obtain the following corollary.

\begin{corollary}\label{cor:spiral_domain_Assouad_change} 
There exists a conformal map $f : \D \to \Omega$ onto a simply connected domain $\Omega \subsetneq \mathbb{C}$ and a diameter $\ell \subset \D$ such that $\dim_A \ell  = 1$ and $\dim_A f(\ell) = 2$. 
\end{corollary}

Theorem 1.2 in \cite{ChrontsiosTyson2023} implies that quasiconformal maps between domains in $\R^n$ preserve the Assouad dimension for compact subsets of the domain with Assouad dimension $n$. Moreover, global quasiconformal maps of $\R^n$ preserve the Assouad dimension of arbitrary sets of Assouad dimension $n$. Corollary \ref{cor:spiral_domain_Assouad_change} shows that the assumption of compactness in the former conclusion, and the assumption that the map is globally defined in the latter conclusion, are both necessary.

\subsection{Structure of the paper} 
Section~\ref{sec:background} provides the necessary definitions and background. We recall the definitions of the Nagata dimension, quasisymmetric, quasiconformal, and quasiregular mappings, together with their basic properties. This section also includes examples illustrating when metric spaces have Nagata dimension zero or one. Section~\ref{sec:analytic_and_rational_maps_ndim} is devoted to results establishing the invariance of Nagata dimension. In subsection \ref{subsec:analytic_maps_ndim}, we prove Theorem~\ref{mainthm1:polynomial_preserve_ndim}, asserting that analytic maps preserve the Nagata dimension of compact subsets of their domains. As a consequence, planar quasiregular maps also preserve the Nagata dimension of compact subsets. Subsection~\ref{sec:rational_maps_ndim} discusses the invariance of Nagata dimension under meromorphic maps and contains the proof of Theorem~\ref{mainthm2:rational_preserve_ndim}. Section~\ref{sec:entire_maps_ndim} contains the proofs of Theorem~\ref{thm:entire_func_increase_ndim} and Theorem~\ref{thm:entire_genus_1_decrease_ndim}, which describe the increase and decrease of Nagata dimension under entire functions. Section~\ref{sec:porosity}
explores the relationship between porosity,
Nagata and Assouad dimensions, and quasiconformal mappings,
and includes the proof of
Theorem~\ref{thm:spiral_domain_porosity_change}. Finally, section~\ref{sec:open_questions} concludes the paper with several open questions and further directions motivated by our results.

\subsection{Notation}

We denote by $\#S$ the cardinality of a set $S$. In a metric space $(X,d)$, we write $\diam A = \sup \{ d(a,a') :a,a' \in A\}$ for the diameter of $A \subset X$ and $\dist (A,B) = \inf \{ d(a,b):a \in A, b\in B\}$ for the distance between sets $A,B \subset X$. The open ball with center $x$ and radius $r>0$ in $(X,d)$ will be denoted $B(x,r)$, while the corresponding closed ball will be denoted $\overline{B}(x,r)$.

The length of an interval $I \subset \R$ is denoted by $\ell(I) := \sup I - \inf I$.

A set $A$ contained in a domain $\Omega \subset \R^n$ is said to be {\it compactly contained} in $\Omega$ if its closure $\overline{A}$ is a compact subset of $\Omega$. We write $A \Subset \Omega$ to indicate that $A$ is compactly contained in $\Omega$.

\subsection*{Acknowledgments}
We thank Guy C. David, Sylvester Eriksson-Bique, Urs Lang, and Vyron Vellis for helpful conversations on the topic of Nagata dimension and its distortion properties. The construction in Theorem \ref{thm:spiral_domain_porosity_change} was suggested to the authors by Sylvester Eriksson-Bique, and we are grateful to him for permission to include it in this paper.

The first author gratefully acknowledges the support and hospitality of the Hausdorff Institute for Mathematics, Bonn. This work was funded by the Deutsche Forschungsgemeinschaft (DfG, German Research Foundation) under the German Excellence Strategy – EXC-2047/1 – 390685813 and the McNamara Grant by World Bank (2025). The second author is supported by the Simons Foundation under Simons Collaboration Grant \#852888.

\section{Background}\label{sec:background}

\subsection{Nagata dimension}
Let $\mathcal{B} = \{B_i\}_{i \in I}$ be a family of subsets of a metric space $(X,d)$. The family $\mathcal{B}$ is called \emph{$s$-bounded} for some constant $s \ge 0$ if $\diam B_i \le s$ for all $i \in I$. For $s > 0$, the \emph{$s$-multiplicity} of $\mathcal{B}$ is the infimum of all integers $n \ge 0$ such that every subset of $X$ with diameter $\le s$ meets at most $n$ members of the family.

\begin{definition}\label{def:nagata_dim}
For a metric space $X$, the \emph{Nagata dimension} (or \emph{Assouad--Nagata dimension}) $\ndim X$ is the infimum of all integers $n$ with the following property: There exists a constant $c > 0$ such that for all $s > 0$, the space $X$ admits a $cs$-bounded covering with $s$-multiplicity at most $n + 1$.
\end{definition}

Nagata dimension was first introduced and named by Assouad in \cite{AssouadOriginal}. 

\begin{theorem}[Finite stability, Theorem 2.7 in \cite{LangSchlichenmaier2005}]\label{thm:finite_stability_ndim}
Let $X = Y \cup Z$. Then $\ndim X = \sup\{\ndim Y, \ndim Z\}$.
\end{theorem}

However, Nagata dimension is not countably stable. For instance, $\ndim \Z = 1$ whereas $\ndim \{x\} = 0$ for each singleton $\{x\}$.

\begin{theorem}[Invariance under completion, \cite{LangSchlichenmaier2005}]\label{thm:completion_invariance_ndim}
Let $(X,d)$ be an incomplete metric space, and let $\overline{X}$ denote the metric completion of $X$. Then $\dim_N X = \dim_N \overline{X}$.
\end{theorem}

\begin{lemma}[Exponential sequence]\label{lemma:exp_sequence}
For $\lambda>1$, let $X = \{\lambda^n \,:\, n \in \N\} \subset \R$ with the Euclidean metric. Then $\ndim X = 0$. In particular, for $c = \tfrac{\lambda}{\lambda - 1}$ and any $s > 0$, the space $X$ admits a $cs$-bounded covering with $s$-multiplicity $1$.
\end{lemma}

\begin{proof}
Fix $s > 0$. Choose $N \in \N$ such that 
\[
    \lambda^{N-1}(\lambda - 1) \le s < \lambda^{N}(\lambda - 1).
\]
Then $|\lambda^{n+1} - \lambda^n| > s$ for all $n \ge N$, and it follows that any subset of $\R$ of diameter at most $s$ can contain at most one of the points $\{\lambda^n\}_{n \ge N}$. The cover
\[
\mathcal{B}_s := \Big\{ \{e, e^2, \ldots, e^N\}, \{e^{N+1}\}, \{e^{N+2}\}, \ldots \Big\}
\]
is $cs$-bounded with $c = \tfrac{\lambda}{\lambda - 1}$. Since every subset of diameter at most $s$ meets at most one element of $\mathcal{B}_s$, the $s$-multiplicity of this cover is $1$. Hence $\ndim X = 0$.
\end{proof}

\begin{definition}[$s$–chains and $s$–chain components]\label{def:s-chain}
Let $(X,d)$ be a metric space and let $s>0$. A finite sequence of points $x_0,x_1,\dots,x_N$ in $X$ is called an \emph{$s$–chain} if $d(x_i,x_{i+1})\le s$ for every $i$. Two points $x,y\in X$ are said to be \emph{$s$–chain connected} if there exists an $s$–chain $x=x_0,\dots,x_N=y$. The notion of $s$-chain connectedness is an equivalence relation, and maximal $s$-chain connected subsets of $X$ are called \emph{$s$–components} of $X$.
\end{definition}

\begin{remark}
It is straightforward to verify that a metric space $(X,d)$ has Nagata dimension zero if and only if there exists a constant $C_X \ge 1$ such that for every $s>0$, each $s$–chain component of $X$ has diameter at most $C_X s$.
\end{remark}

\begin{lemma}[growing chains]\label{lemma:chain-growth}
Let $(X,d)$ be a metric space. Suppose that there exists $\delta>0$ and a sequence of finite $\delta$–chains $C_n=\{x^n_0,\dots,x^n_{N_n}\}\subset X$ such that $\diam C_n \stackrel{n \to \infty}{\longrightarrow} \infty$. Then $\ndim X \ge 1$. In particular, if $X\subset\R$, then $\ndim X =1$.
\end{lemma}

\begin{proof}
If $\ndim X = 0$, then there exists $C_X \ge 1$ such that every $s$–chain component of $X$ has diameter $\le C_X s$. Fix $s=\delta$. Each $C_n$ is a $\delta$–chain, and hence is contained in a $\delta$–chain component. Thus $\diam C_n \le C_X \delta$ for all $n$, which contradicts the assumption on $\diam C_n$. Hence $\ndim X \ge 1$. If $X\subset\R$, then $\ndim X \le 1$ and hence $\ndim X =1$.
\end{proof}

\begin{corollary}\label{cor:growing_arithmetic_patches}
Let $X\subset\R$. If for some $\delta>0$ there exist finite arithmetic progressions $\{a_n,a_n+\delta,a_n+2\delta,\dots,a_n+M_n\delta\} \subset X$ with $M_n\to\infty$, then $\ndim X=1$.
\end{corollary}

Recall that the {\it Assouad dimension} of a metric space $(X,d)$ is the infimum of values $s>0$ for which there exists $C>0$ so that for any $x \in X$ and any $0<r<R$, the ball $B(x,R)$ can be covered by at most $C(R/r)^s$ balls of radius $r$. A metric space $(X,d)$ has finite Assouad dimension if and only if it is doubling. As noted in the introduction,
\begin{equation}\label{eq:n-to-a}
\dim_N X \le \dim_A X
\end{equation}
for any metric space $X$ \cite[Theorem 1.1]{DonneRajala2015}.

\subsection{Quasisymmetric maps and ultrametric spaces}

\begin{definition}\label{def:QS}
Let $\eta : [0, \infty) \to [0, \infty)$ be a homeomorphism. A homeomorphism  $f : (X, d_X) \to (Y, d_Y)$ between metric spaces is \emph{$\eta$-quasisymmetric} if for all $x, y, z \in X$, $x \ne z$, we have
\[
\frac{d_Y(f(x), f(y))}{d_Y(f(x), f(z))} \le \eta\!\left( \frac{d_X(x, y)}{d_X(x, z)} \right).
\]
A homeomorphism $f : (X, d_X) \to (Y, d_Y)$ is \emph{quasisymmetric} if it is $\eta$-quasisymmetric for some homeomorphism $\eta$, and metric spaces $X$ and $Y$ are {\it quasisymmetrically equivalent} if there exists a quasisymmetric map $f:X \to Y$.
\end{definition}

Foundations for the theory of quasisymmetric maps in metric spaces were laid by Tukia and V\"ais\"al\"a \cite{TukiaVaisala1980}. For more background see \cite[Chapters 10-12]{Heinonen2001LecturesonAnalysis}.

\begin{theorem}[Theorem 1.2, \cite{LangSchlichenmaier2005}]\label{thm:QS_invariance_ndim}
If $X$ and $Y$ are quasisymmetrically equivalent, then $\ndim X = \ndim Y$.
\end{theorem}

Bi-Lipschitz maps are a particular example of quasisymmetric maps. Recall that $f : (X,d_X) \to (Y,d_Y)$ is \emph{$L$-bi-Lipschitz}, $L \ge 1$, if
\[
L^{-1} \, d_X(x, y) \le d_Y(f(x), f(y)) \le L \, d_X(x, y)
\]
for all $x, y \in X$. Every $L$-bi-Lipschitz mapping is $\eta$-quasisymmetric with $\eta(t) = L^2 t$.

\begin{definition}\label{def:ultrametric_spaces}
    A metric space $(X, d)$ is called \emph{ultrametric} if for all $x, y, z \in X$ we have $d(x, z) \le \max\{d(x, y), d(y, z)\}$.
\end{definition}

\begin{theorem}[Theorem 3.3 in {\cite{BrodskiyDydakHigesMitra2007}}] \label{thm:ndim_ultrametric_brodskiyetal}
If a metric space $(X, d)$ has Nagata dimension zero, then there exists an ultrametric $\rho$ on $X$ such that the identity map $\mathrm{id} : (X, d) \to (X, \rho)$ is bi-Lipschitz.
\end{theorem}

In particular, since bi-Lipschitz maps are quasisymmetric, $\ndim(X, \rho) =0$. Moreover, if $X$ is bi-Lipschitz equivalent to an ultrametric space, then $\ndim X = 0$. For details, see \cite{BrodskiyDydakHigesMitra2007}.

The following lemma generalizes Lemma \ref{lemma:exp_sequence}. Here we give a different proof which relies on the characterization of Nagata dimension zero in terms of ultrametrics.

\begin{lemma}\label{lemma:ratio>1_ndim_0}
Let $(X,d)$ be a metric space, fix a basepoint $o\in X$ and $\lambda>1$, and let  $Z:=\{x_k \, : \, k\in\N\}\subset X$ be a discrete sequence such that 
\[
    d(x_{k+1},o)\;\ge\; \lambda \, d(x_k,o)
\]
for all $k \in \N$. Then $\ndim Z = 0$.
\end{lemma}

\begin{proof}
Set $r_k:=d(x_k,o)$, so $r_{k+1}\ge \lambda \,r_k$ and hence $(r_k)$ is strictly increasing. Define $d_U:Z\times Z\to[0,\infty)$ by
\[
d_U(x_m,x_n)\;:=\;\begin{cases}
0,& m=n,\\[2pt]
r_{\max\{m,n\}},& m\neq n.
\end{cases}
\]
We claim that $d_U$ is an ultrametric on $Z$. 
It suffices to verify the triangle inequality, so fix $m,n,p$ and assume without loss of generality that $m\le n$.
If $p\ge n$ then $d_U(x_m,x_n)=r_n\le r_p=\max\{d_U(x_m,x_p),d_U(x_p,x_n)\}$. 
If $m\le p<n$ then $d_U(x_m,x_n)=r_n=\max\{r_p,r_n\}=\max\{d_U(x_m,x_p),d_U(x_p,x_n)\}$. 
Finally, if $p\le m\le n$ then $d_U(x_m,x_p)=r_m\le r_n=d_U(x_m,x_n)$ and $d_U(x_p,x_n)=r_n$, 
so again $d_U(x_m,x_n)\le\max\{d_U(x_m,x_p),d_U(x_p,x_n)\}$. Thus $d_U$ is an ultrametric.

Next we show that $(Z,d)$ and $(Z,d_U)$ are bi-Lipschitz equivalent.  Assume again that $m<n$. Then $r_n-r_m \le d(x_m,x_n) \le r_n+r_m$ by the triangle inequality in $(X,d)$, and since $r_m\le r_n/a$ we conclude that
\[
\bigl(1-\frac{1}{\lambda}\bigr)r_n \, \le \, d(x_m,x_n) \, \le \, \bigl(1+\frac{1}{\lambda}\bigr)r_n.
\]
Since $d_U(x_m,x_n)=r_n$ we obtain
\[
    \frac{\lambda}{\lambda+1}\,d(x_m,x_n) \, \le \, d_U(x_m,x_n)
    \, \le \,
    \frac{\lambda}{\lambda-1}\,d(x_m,x_n).
\]
Therefore the identity map $\mathrm{id}:(Z,d)\to (Z,d_U)$ is bi-Lipschitz. Since $Z$ is bi-Lipschitz equivalent to an ultrametric space, it follows that $\ndim Z =0$.
\end{proof}
    
\subsection{Quasiconformal and quasiregular maps}

\begin{definition}\label{def:QC}
A homeomorphism $f : (X, d_X) \to (Y, d_Y)$ between metric spaces is \emph{$H$-quasiconformal} for some $H \ge 1$ if  
\begin{equation}\label{eqn:metrically_qc_H_f}
   \limsup_{r \to 0} \frac{\sup\{d_Y(f(x), f(y)) : d_X(x, y) \le r\}}{\inf\{d_Y(f(x), f(z)) : d_X(x, z) \ge r\}} \le H
\end{equation}
for every $x \in X$.
\end{definition}
    
\begin{remark}\label{rem:qcqs}
If $f : \Omega \to \Omega'$ is a homeomorphism between domains in $\R^n$, $n \ge 2$, then $f$ is quasiconformal if and only if it is locally quasisymmetric. Moreover, if $\Omega = \Omega' = \R^n$ then $f$ is quasiconformal if and only if it is quasisymmetric. Orientation-preserving and $1$-quasiconformal homeomorphisms of planar domains are precisely the conformal mappings.
\end{remark}

In general, quasiconformal maps need not preserve the Nagata dimension of arbitrary (non-compact) subsets. For example, the conformal map $f(z) = e^z$ defined on $\Omega = \{\,x + \bi y : x \in \R,\, 0 < y < 2\pi \}$ maps the discrete set $\N \subset \Omega$ of Nagata dimension one to the set $\{e^n : n \in \N\}$ of Nagata dimension zero. However, Lemma \ref{rem:qcqs} and Theorem \ref{thm:QS_invariance_ndim} immediately imply the following result.

\begin{proposition}\label{lemma:QC_invariance_ndim_compact}
Let $f:\Omega \to \Omega'$ be a quasiconformal map between domains in $\R^n$, $n \ge 2$. Then $\dim_N E = \dim_N f(E)$ for every compact set $E \subset \Omega$.
\end{proposition}

\begin{proof}
Since quasiconformal maps in $\R^n$ are locally quasisymmetric, the result follows from the finite stability of Nagata dimension (Theorem~\ref{thm:finite_stability_ndim}) and its invariance under quasisymmetric maps (Theorem~\ref{thm:QS_invariance_ndim}).
\end{proof}

\begin{definition}\label{def:analytic_QR}
Let $\Omega \subset \R^n$ be a domain and let $f : \Omega \to \R^n$ be a continuous mapping in the Sobolev space $W^{1,n}_{\mathrm{loc}}(\Omega, \R^n)$. We say that $f$ is \emph{$K$--quasiregular}, for some $K \ge 1$, if it satisfies the distortion inequality
\begin{equation}\label{eqn:analytic_QR}
            |Df(x)|^n \le K\, J_f(x)
            \qquad \text{for almost every } x \in \Omega,
        \end{equation}
where $|Df(x)| = \sup_{|v|=1} |Df(x)v|$ denotes the operator norm of the differential and $J_f(x)$ is the Jacobian determinant of $f$.
\end{definition}

Quasiregular mappings admit a metric characterization akin to Definition \ref{def:QC}; see \cite[Theorem 6.2]{RickmanBookQRMaps}. On the other hand, quasiconformal maps of Euclidean domains are homeomorphisms which are also quasiregular in the sense of Definition \ref{def:analytic_QR}. In the planar case, quasiregular maps have a particularly simple structure: each such map can be expressed as the composition of a quasiconformal homeomorphism and a holomorphic function. In fact every analytic function is $1$-quasiregular.

\begin{theorem}[Sto{\"\i}low factorization, \cite{LehtoVirtanen1973}]\label{thm:Stoilow}
Let $\Omega \subset \C$ be a domain and let $f : \Omega \to \C$ be quasiregular. Then there exists a quasiconformal homeomorphism $h : \Omega \to \Omega'$ onto a domain $\Omega' \subset \C$ and an analytic function $g : \Omega' \to \C$ so that $f = g \circ h$.
\end{theorem}

\subsection{Spherical metric on the Riemann sphere and quasi-M\"obius mappings}

It is convenient to work on the extended complex plane, $\hatc := \C \cup \{\infty\}$ equipped with the spherical metric
\begin{equation*}\label{eqn:spherical_metric}
    q(z,w) := \frac{|z-w|}{\sqrt{(1+|z|^2)(1+|w|^2)}}
\end{equation*}
and
\[
    q(z,\infty) = q(\infty,z) = \frac{1}{\sqrt{1+|z|^2}}, \qquad q(\infty,\infty) = 0.
\]
We introduce the {\it inversion} $\iota: \C \setminus\{0\} \to \C \setminus \{0\}$ given by $\iota(z) = \tfrac1z$, and we extend $\iota$ to a map $\iota:\hatc \to \hatc$ by $\iota(\infty) = 0$ and $\iota(0) = \infty$. The metric $q$ is invariant under inversion, $q(\iota(z),\iota(w)) = q(z,w)$, and induces the standard conformal structure on the Riemann sphere $\hatc$.

The spherical metric is bi-Lipschitz equivalent to the Euclidean metric on bounded subsets of $\C$: if $E \subset \C$ is bounded, then $c_E^{-1}|z-w| \le q(z,w) \le |z-w|$ for all $z,w \in E$, where $c_E = 1 + \sup_{z \in E} |z|^2$.

Recall that the M\"obius transformations of $\hatc$ preserve the cross-ratio of quadruples of elements of $\hatc$. In \cite{VaisalaQM1984}, V\"ais\"al\"a introduced the notion of quasi-M\"obius map in the metric space setting.

\begin{definition}\label{def:QM}
Let $\eta : [0,\infty) \to [0,\infty)$ be a homeomorphism. A homeomorphism $f : (X, d_X) \to (Y,d_Y)$ is \emph{$\eta$-quasi-M\"obius} if for all quadruples of distinct points $x,a,b,c \in X$, we have
\[
  \frac{d_Y(f(x),f(a))\,d_Y(f(b),f(c))}{d_Y(f(x),f(b))\,d_Y(f(a),f(c))} 
  \le \eta\!\left(
  \frac{d_X(x,a)\,d_X(b,c)}{d_X(x,b)\,d_X(a,c)}
  \right).
\]
\end{definition}

Every quasisymmetric map is quasi-M\"obius, and conversely, every quasi-M\"obius map between bounded metric spaces (or which preserves the point at infinity in the one point extensions, cf.\ \cite[Theorem 3.10]{VaisalaQM1984}) is quasisymmetric.

\begin{theorem}[Theorem 1.1 in \cite{Xie2008}]\label{thm:QM_invariance_ndim}
Let $f:(X,d_X) \to (Y,d_Y)$ be quasi-M\"obius. Then $\dim_N X = \dim_N Y$.
\end{theorem}

Note that the cross-ratio of a quadruple of elements of $\C$ is the same whether we compute using the Euclidean metric $d_E$ or the spherical metric $q$. Thus $\id:(\C,d_E) \to (\C,q)$ is M\"obius and preserves Nagata dimensions of sets.

In \cite{VaisalaQM1984}, V\"ais\"al\"a further explores the relationship between quasiconformality and quasisymmetry (or the quasi-M\"obius condition) in the setting of domains. We state a special case of \cite[Theorem 5.4]{VaisalaQM1984}.

\begin{theorem}[V\"ais\"al\"a, Theorem 5.4 in \cite{VaisalaQM1984}]\label{th:qcqm}
Let $\Omega$ and $\Omega'$ be uniform domains in $\hatc$. Then every quasiconformal homeomorphism $f: \Omega \to \Omega'$ is quasi-M\"obius. 
\end{theorem}

We omit the definition of uniform domain here, but refer to \cite[\S4.6]{VaisalaQM1984} for details. In the case when $\Omega \subset \C$ is a Jordan domain in $\hatc$, the domain $\Omega$ is uniform if and only if the boundary of $\Omega$ in $\hatc$ is a quasicircle \cite[Corollary 2.33]{ms:injectivity}.

\section{Invariance of Nagata dimension}\label{sec:analytic_and_rational_maps_ndim}

This section contains results establishing the invariance of Nagata dimension in various settings. Subsection \ref{subsec:analytic_maps_ndim} focuses on the invariance of Nagata dimension under analytic functions and particularly under polynomials, while 
subsection \ref{sec:rational_maps_ndim} tells a parallel story for meromorphic functions and rational functions.

\subsection{Nagata dimension and analytic functions}\label{subsec:analytic_maps_ndim}

In this section, we prove Theorem \ref{mainthm1:polynomial_preserve_ndim} on the invariance of Nagata dimension by non-constant complex polynomials. Towards this end, we first prove that analytic functions preserve Nagata dimension of compact subsets in their domain.
    
\begin{lemma}\label{lemma:zn_weakly_qs}
Let $g(z)=z^n$ for $n\in\N$. Fix $\theta\in(0,2\pi/n)$ and $\alpha>0$, and set
$$
            \Omega_{\theta,\alpha} := \{z\in\C :\ \alpha < \arg z < \alpha + \theta\}.
$$
Then $g:\overline{\Omega_{\theta,\alpha}} \to \overline{g(\Omega_{\theta,\alpha})}$ is quasi-M\"obius.
\end{lemma}
    
\begin{proof}
The sector $\Omega_{\theta,\alpha}$ is a quasidisk \cite[Example 1.4.2]{gh:quasidisk} and hence its boundary in $\hatc$ is a quasicircle. Thus $\Omega_{\theta,\alpha}$ is a uniform domain. Since $\theta <\tfrac{2\pi}{n}$ the map $g$ is conformal on $\Omega_{\theta,\alpha}$, and hence $1$-quasiconformal. By Theorem \ref{th:qcqm}, $g$ is quasi-M\"obius.
\end{proof}

\begin{corollary}\label{cor:zn_preserves_ndim}
        Let $g(z)=z^n$ for some $n\in\N$. Then $\ndim X = \ndim g(X)$ for every subset $X\subset\C$.
\end{corollary}
    
\begin{proof}
 We appeal to the finite stability theorem~\ref{thm:finite_stability_ndim} for the Nagata dimension. Partition the plane into angular sectors
 $$
  P_j := \{ z\in\C : \tfrac{j \pi}{n} < \arg z < \tfrac{(j+1)\pi}{n} \}, \qquad j=0,1,\ldots,2n-1.
 $$
Note that $P_j$ is equal to the domain $\Omega_{\theta,\alpha}$ in Lemma \ref{lemma:zn_weakly_qs} for a suitable $\alpha$ and for $\theta = \tfrac\pi{n}$. Since $\theta < \tfrac{2\pi}{n}$ the hypotheses of the lemma are satisfied and hence $g|_{P_j}$ is quasi-M\"obius and preserves the Nagata dimension of arbitrary subsets of $P_j$. Since $\C = \bigcup_{j=0}^{2n-1} P_j$ is a finite union, the finite stability of Nagata dimension implies that
        $$
        \ndim X
        = \max_{0\le j\le n-1} \ndim (X\cap P_j)
        = \max_{0\le j\le n-1} \ndim g(X\cap P_j)
        = \ndim g(X).
        $$
for any $X\subset\C$. The proof is complete.
\end{proof}

    \begin{theorem}\label{thm:analytic_preserve_ndim_compact}
        Let $f: \Omega\to \C$ be an analytic map where $\Omega$ is a domain. If $E\subset\Omega$ is compact, then $\ndim E = \ndim f(E)$.
    \end{theorem}
    
    \begin{proof}
         Since $f$ is analytic, the set of its critical points in $E$ is finite. If $f$ has no critical points in $E$, then $f$ is conformal on $E$, and hence by Lemma~\ref{lemma:QC_invariance_ndim_compact} $\ndim E = \ndim f(E)$ in this case.
    
        Suppose now that $f$ has critical points $c_1,\dots,c_m$ in $E$. Since $f$ is analytic, for each $c_j$ there exists a radius $\varepsilon_j>0$ and conformal maps $\phi_j, \psi_j$ such that
        \begin{equation*}
            \psi_j\circ f\circ\phi_j^{-1}(z) = z^{k_j}
        \end{equation*}
        for all $z \in B_{c_j}$ where $B_{c_j}$ is a ball of radius $\epsilon_j$ around $c_j$ and $k_j$ is the multiplicity of $f$ at $c_j$. 
    
        Since Nagata dimension is preserved under metric completion, 
        $
            \ndim \overline{B_{c_j} \cap E} = \ndim B_{c_j} \cap E.
        $
        Each $\phi_j$ and $\psi_j^{-1}$ is conformal and hence quasiconformal on compact subsets of their domains and by Lemma~\ref{lemma:QC_invariance_ndim_compact}, they both preserve Nagata dimension. 
    
        By Corollary \ref{cor:zn_preserves_ndim}, $z^{k_j}$ preserves Nagata dimension for each $j$ and since $f = \psi_j^{-1} \circ z^{k_j} \circ \phi_j$ on $B_{c_j}$, we conclude that
        $$ 
            \ndim B_{c_j} \cap E = \ndim f(B_{c_j} \cap E), \qquad j=1,\dots,m. 
        $$    
        Outside the critical neighborhoods, that is, on
        $$ 
        E' := E \setminus \bigcup_{j=1}^m B_{c_j},
        $$
        $f$ is conformal (and hence quasiconformal). Therefore, $f$ preserves Nagata dimension on each of the finitely many sets $E'$ and $B_{c_j} \cap E$, $j=1,\dots,m$. The conclusion follows from Theorem~\ref{thm:finite_stability_ndim}.
    \end{proof}

\begin{corollary}\label{cor:analytic_map_any_subset_compact}
 Let $f: \Omega\to \C$ be an analytic map where $\Omega$ is a domain. If $A\Subset\Omega$ then $\ndim A = \ndim f(A)$.
\end{corollary}
    
By Sto{\"\i}low's factorization theorem~\ref{thm:Stoilow}, we immediately obtain the following result.
    
\begin{corollary}
Let $g: \Omega \to \C$ be quasiregular. If $A \Subset \Omega$ then $\ndim A = \ndim g(A)$.
\end{corollary}

We conclude this section with the proof of Theorem~\ref{mainthm1:polynomial_preserve_ndim}.

\begin{proof}[Proof of Theorem~\ref{mainthm1:polynomial_preserve_ndim}]
    Let $p(z) = a_nz^n+a_{n-1}z^{n-1}+ \cdots+a_1 z+a_0$ and define
    \begin{equation*}
        r(z):=\iota \circ p\circ \iota(z)=\frac{1}{p(1/z)} = \frac{z^n}{1+a_{n-1}z+\cdots+a_1 z^{n-1}+ a_0z^n}.
    \end{equation*}
Since $p$ is a polynomial, $r$ is analytic at $0$. Hence there exists $\rho>0$ such that $r$ is analytic on the closed disk $\overline{B(0,\rho)}$. By Theorem~\ref{thm:analytic_preserve_ndim_compact}, $r$ preserves the Nagata dimension of subsets of $\overline{B(0,\rho)}$. Moreover, $\iota$ is quasi-M\"obius and thus preserves Nagata dimension by Theorem \ref{thm:QM_invariance_ndim}. Let  $X_o = X \cap \{z : |z|> 1/\rho\}$, then $\ndim X_o = \ndim r \circ \iota (X_o) = \ndim \iota \circ r \circ \iota (X_o)$. Invoking Theorem~\ref{thm:analytic_preserve_ndim_compact} again for $X_i:= X\setminus X_o$, we have $\ndim X_i = \ndim p(X_i)$. By finite stability, we have
    \begin{eqnarray*}
        \ndim X &= &\max\{\ndim X_o, \ndim X_i\} \\
        & = & \max\{\ndim p(X_o), \ndim p(X_i)\} = \ndim p(X)
    \end{eqnarray*}
and the proof is complete.
\end{proof}

\subsection{Nagata dimension and rational functions}\label{sec:rational_maps_ndim}

In this subsection we discuss the invariance properties of Nagata dimension under non-constant meromorphic functions and particularly under rational functions.

\begin{lemma}\label{lemma:meromorphic_around_poles}
Let $f(z)$ be a meromorphic function defined on a domain $\Omega \subset \hatc$.
To each pole of $f$ there corresponds a neighborhood $U$ such that the Nagata dimension of any subset of $U$ is preserved under $f$.
\end{lemma}

\begin{proof}
    Let $p$ be a finite pole of $f(z)$. Since poles are isolated, there exists a precompact neighborhood $U$ of $p$ so that $1/f(z) = (z-p)^n h(z)$ where $n\ge 1$ and $h(z)$ is non-vanishing and holomorphic on $U$. Let $K = \Bar{U} \subset \Omega$ denote the  compact closure of $U$. Since $f(z) = \iota(\frac{1}{f(z)})$, Theorem~\ref{thm:QM_invariance_ndim} and Corollary~\ref{cor:analytic_map_any_subset_compact} imply that $f$ preserves Nagata dimension. 

   If $p=\infty$, choose the local coordinate $w=1/z$ and define  $g(w):=f(1/w)$ on a punctured neighborhood of $w=0$. Since $f$ is meromorphic at $\infty$, $g$ is meromorphic at $0$. By the finite pole case applied to $g$ at $w=0$, there exists a disk $V$ about $0$ such that $\ndim B=\ndim g(B)$ for all $B\subset V$. Let $U := \{z\in\Omega : 1/z\in V\}$. For any $A\subset U$, setting $B:=\{1/z : z\in A\}$ gives
   \[
    \ndim A = \ndim B = \ndim g(B) = \ndim f(A),
   \]
   by Theorem~\ref{thm:QM_invariance_ndim} since $\iota$ is quasi-M\"obius.
\end{proof}

\begin{lemma}\label{lemma:meromorphic_compact_ndim}
    Let $f(z)$ be a non-constant meromorphic function on $\Omega \subset \hatc$ and let $A \Subset \Omega$. Then $\ndim A = \ndim f(A)$.
\end{lemma}

\begin{proof}
Since Nagata dimension is invariant under closure, it suffices to assume that $A$ is compact.

If $f$ has no poles in $A$, then $f$ is holomorphic on a neighborhood of $A$, and the claim follows from Theorem~\ref{thm:analytic_preserve_ndim_compact}.

Otherwise, since the poles of a meromorphic function are isolated and $A$ is compact, $A$ contains only finitely many poles. Let $p_1,\dots,p_k$ be these poles (allowing $p_j=\infty$ if $\infty\in\Omega$). By Lemma \ref{lemma:meromorphic_around_poles}, there exists a neighborhood $U_j$ for each $p_j$ such that $f(z)$ preserves Nagata dimension of any subset of $U_j$. We write
\begin{equation*}\label{eqn:meromorphic_decompose_E}
A= \bigcup_{j=1}^{k} (A \cap U_j) \cup \Bigl(A \setminus \bigcup_{j=1}^{k} U_j\Bigr).
\end{equation*}
For each $j$, $f$ preserves the Nagata dimension of $A \cap U_j$ by Lemma~\ref{lemma:meromorphic_around_poles}. On $A \setminus \bigcup_{j=1}^{k} U_j$, the function $f$ is analytic, and Theorem~\ref{thm:analytic_preserve_ndim_compact} implies that $f$ preserves the Nagata dimension there as well. Theorem~\ref{thm:finite_stability_ndim} applied to the above decomposition of $A$ yields $\ndim A = \ndim f(A)$.
\end{proof}

\begin{proof}[Proof of Theorem~\ref{mainthm2:rational_preserve_ndim}]
Let $X$ be any subset of $\C$. If $X$ is bounded, then $ \ndim X = \ndim r(X)$ by Lemma \ref{lemma:meromorphic_compact_ndim}.

Suppose that $X$ is unbounded. Define $g(w):=r(1/w) = (r \circ \iota)(w)$, which is again rational (and hence meromorphic on $\hat{\C}$). Choose a bounded neighborhood $U$ of $0$ in $\C$. By  Lemma~\ref{lemma:meromorphic_compact_ndim}, $g$ preserves Nagata dimension of any subset of $U$.
Let
    \begin{eqnarray*}
         U' := \iota(U) = \{z\in\C : 1/z \in U\}.
    \end{eqnarray*}

    By Theorem~\ref{thm:QM_invariance_ndim}, 
    \begin{equation*}
        \ndim A' = \ndim \iota(A') = \ndim g(\iota(A')) = \ndim r(A')
    \end{equation*}
    for every subset $A'\subset U'$. Hence $r$ preserves Nagata dimension of subsets of $X\cap U'$.
    The complement $V := \hat{\C} \setminus U'$ is bounded in $\C$, whence $\ndim(X \cap V) = \ndim(r(X\cap V))$ by Lemma \ref{lemma:meromorphic_compact_ndim}. Since $X = (X \cap U') \cup (X \cap V)$,
    finite stability of the Nagata dimension (Theorem~\ref{thm:finite_stability_ndim}) implies $\ndim(r(X)) = \ndim(X)$.
\end{proof}

\section{Entire functions and Nagata dimension}\label{sec:entire_maps_ndim}

In view of Theorem~\ref{mainthm1:polynomial_preserve_ndim}, it is natural to ask whether there exist other classes of entire functions that preserve the Nagata dimension of arbitrary subsets of~$\C$. In this section we show that this is not the case, by establishing several dimension distortion results for entire functions.

\subsection{Increase in Nagata dimension under entire and meromorphic functions}

In this short section we give the proofs of Theorem~\ref{thm:entire_func_increase_ndim} and Corollary \ref{cor:polynomial-characterization}. We also discuss extensions of these results to meromorphic functions.

\smallskip

By Picard's great theorem \cite[Ch.\ 10, §4]{RemmertComplexFunctionTheory}, in every neighborhood of $\infty$, $f$ attains every complex value infinitely many times, with at most one exception. Let $W = \C \setminus f(\C)$ and note that $\#W \le 1$.

Let $Y \subset \C$ and let $\{y_n\}_{n\in \N} \subset \C\setminus W$ be a countable set which is dense in $Y$. Choose a strictly increasing sequence of radii $(R_i)_{i\ge 1}$ with $R_1 = 1$ and $R_{i+1} > 2 R_i$ together with values $(x_i)$ so that $\tfrac12 R_{i+1} > |x_i| > R_i$ and $f(x_i) = y_i$. Then the countable set $X = \{x_i\} \subset \Omega$ satisfies $|x_{i+1}| >  R_{i+1} > 2|x_i| > |x_i | > R_i$. By Lemma \ref{lemma:ratio>1_ndim_0}, $\ndim(X) = 0$. Furthermore, $f(X) = \{y_n\}$ is dense in $Y$.

This completes the proof of Theorem~\ref{thm:entire_func_increase_ndim}. Corollary \ref{cor:polynomial-characterization} is an immediate consequence of Theorem \ref{mainthm1:polynomial_preserve_ndim} and Theorem~\ref{thm:entire_func_increase_ndim}.

\begin{remark}
The same conclusion holds for functions $f:\C \to \hat{\C}$ which are meromorphic with an essential singularity at $\infty$, to wit, for any $Y \subset \hat{\C}$ there exists $X \subset \C$ with $\dim_N X = 0$ and $f(X)$ is dense in $Y$. Here one uses the meromorphic version of Picard's great theorem \cite[Chapter XVI, Prefatory Remarks and Exercise 2.2]{heins:cft}, which states that a meromorphic function $g$ in a punctured disc $B(z_0,\delta) \setminus \{z_0\}$ with an essential singularity at $z_0$ takes on all values of $\C$ infinitely often, with two exceptions. The proof proceeds as before, applying this version of Picard's great theorem to the function $g(z) = f(\tfrac1z)$. From this conclusion and Theorem~\ref{mainthm2:rational_preserve_ndim} one again quickly derives the following
\end{remark}

\begin{corollary}
Let $f:\C \to \hat{\C}$ be meromorphic. Then $f$ is a nonconstant rational function if and only if $\dim_N X = \dim_N f(X)$ for every $X \subset \C$.
\end{corollary}

\subsection{Two preliminary lemmas}

We state several lemmas which will be used in the proof of Theorem~\ref{thm:entire_genus_1_decrease_ndim}.

\begin{lemma}[Radial growth]\label{lemma:radial-growth}
    Let $q(z)=a_d z^d+\cdots+a_1 z+a_0$ be a nonconstant polynomial of degree $d\ge 1$, with $a_d>0$.
    Then there exist constants $R_0,c_0>0$ (depending only on $q$) such that for all $r\ge R_0$
    and all $h>0$,
    \[ 
    \Re\big(q(r+h)-q(r)\big)\ \ge\ c_0~ r^{~d-1}~ h.
    \]
\end{lemma}

\begin{proof}
    Write
    \[
    q'(z)=d a_d z^{d-1}+\sum_{m=0}^{d-2} c_m z^{m}.
    \]
    For real $t\ge 1$,
    \[
    \Re q'(t)\ \ge d a_d~ t^{d-1}-\sum_{m=0}^{d-2} |c_m|~ t^{m}
    \ \ge d a_d~ t^{d-1}-C~ t^{d-2},
    \]
    where $C:=\sum_{m=0}^{d-2}|c_m|$. Choose
    \[
    R_0\ \ge \max\Big\{1,\ \frac{2C}{d a_d}\Big\}
    \]
and set $c_0:=\tfrac12 d a_d$. Then for all $t\ge R_0$,
    \[
    \Re q'(t)\ \ge \tfrac12 d a_d~ t^{d-1}\ =\ c_0~ t^{d-1}.
    \]
    Define $g(t):=\Re q(t)$ for real $t$. By the mean value theorem, for any $r\ge R_0$ and $h>0$ there exists $\xi\in(r,r+h)$ with
    \[
    \Re\big(q(r+h)-q(r)\big) = g(r+h)-g(r) = g'(\xi)~h = \Re q'(\xi)~h
     \ge c_0~ \xi^{~d-1}~ h \ge c_0~ r^{~d-1}~ h,
    \]
    since $\xi\ge r$. This proves the claim.
\end{proof}

\begin{lemma}[Sublinear growth]\label{lemma:sublinear_growth_zeroes}
Let $A \subset \C$ be a discrete set with $\sum_{\alpha \in A} |\alpha|^{-1} < \infty$. Let $m_R(A) := \#\{\alpha \in A : \alpha \in B(0,R)\}$. Then $\lim_{R \to \infty} m_R(A)/R = 0$.
\end{lemma}

\begin{proof}
For $n\in\N$ let $\mathcal A_n:=\{z\in\C:2^n\le |z|<2^{n+1}\}$ and set $M_n:= \# A\cap\mathcal A_n$. For every $\alpha\in A\cap\mathcal A_n$ we have $|\alpha|<2^{n+1}$, hence
    \[
        \sum_{\alpha\in A\cap\mathcal A_n}\frac{1}{|\alpha|}
        \;\ge\; \frac{M_n}{2^{n+1}}.
    \]
Summing over $n$ gives
    \[
        \infty > \sum_{\alpha \in A}\frac{1}{|\alpha|}
        \;\ge\;\sum_{n=0}^\infty \sum_{\alpha\in A\cap\mathcal A_n}\frac{1}{|\alpha|}
        \;\ge\;\sum_{n=0}^\infty \frac{M_n}{2^{n+1}},
    \]
so the series $\sum_{n\ge0} M_n/2^n$ converges. In particular,
    \begin{equation}\label{eq:m-n}
        2^{-n} M_n \xrightarrow[n\to\infty]{}0.
    \end{equation}
For $n\in \N$ we have $m_{2^{n+1}}(A) = \sum_{k=0}^{n} M_k$. Fix $\varepsilon>0$. By \eqref{eq:m-n} there exists $K$ such that $M_k\le \varepsilon 2^k$ for all $k\ge K$. Hence, for $n\ge K$,
    \[
        m_{2^{n+1}}(A)
    =\sum_{k=0}^{K-1} m_k+\sum_{k=K}^{n} M_k
    \;\le\; C+\varepsilon\sum_{k=K}^{n}2^k
    \;\le\; C+\varepsilon~(2^{n+1}-2^K),
    \]
where $C =\sum_{k=0}^{K-1} M_k$. Dividing by $2^{n+1}$ and letting $n\to\infty$ yields
    \[
        \limsup_{n\to\infty}\frac{m_{2^{n+1}}(A)}{2^{n+1}}\le \varepsilon.
    \]
Since $\varepsilon>0$ was arbitrary,
    \[
        \lim_{n\to\infty}\frac{m_{2^{n+1}}(A)}{2^{n+1}}=0.
    \]
Finally, for general $R>0$ pick $n$ with $2^n\le R<2^{n+1}$. Then
    \[
    0 \le \frac{m_R(A)}{R}\ \le \frac{m_{2^{n+1}}(A)}{2^n} = 2\frac{m_{2^{n+1}}(A)}{2^{n+1}}
    \ \xrightarrow[n\to\infty]{}\ 0.
    \]
    Thus $\lim_{R\to\infty} m_R(A)/R=0$.
\end{proof}

\subsection{Decrease in Nagata dimension under entire functions}

This section is devoted to the proof of Theorem~\ref{thm:entire_genus_1_decrease_ndim}. As the proof is long, we begin with a sketch. 

\medskip

\paragraph{\it Sketch of the proof:} We build a sequence $X = \{z_n\}_{n\in\N}$ iteratively inside dyadic annuli, enumerated $D_n$, whose widths are chosen using the Nagata constant, $c_A$, of the zero set $A$. We also pick a sequence of increasing scales $\{s_n\}$ determined by $c_A$. By using the definition of Nagata dimension for $A$, in each $D_n$ there is a long interval on the real axis that avoids $A$ by at least $\varepsilon s_n$. For fixed $\delta>0$, we insert a $\delta$–spaced arithmetic patch into each such interval. This produces arbitrarily long arithmetic chains in $X$, yielding $\ndim(X)=1$. Next, we prove that along a suitable enumeration, the ratio $|f(z_{k+1})/f(z_k)|$ is eventually $\ge \Lambda >1$, which follows from radial growth of $e^{q(z)}$ combined with control of the infinite product over the zeros. This ensures that $\ndim(f(X))=0$. Note that we first construct a preliminary sequence $\tilde{X}$, and then prune it to keep only those points that satisfy the desired growth properties; the resulting subsequence is the set $X$ used in the argument.

\medskip

\noindent This completes our sketch of the proof, and we now provide a more detailed argument. First, note that if the set $A$ in Theorem~\ref{thm:entire_genus_1_decrease_ndim} has Nagata dimension one, then we may choose $X = A$ and observe that $\dim_N F(X) = 0 \ne \dim_N X$ (since $F(X) = \{0\}$). We may thus assume without loss of generality that $\dim_N(A) = 0$.

We begin by constructing the relevant set $X$.  Let $q(z)=a_d z^d+\cdots+a_0$ with $a_d\neq0$. Applying a rotation if necessary, we may assume that $a_d>0$. 

Since $\ndim(A) = 0$, there exists $c_A \ge 1$ so that for any $s >0$ there exists a $c_As$-bounded and $s-$disjoint collection $\mathcal{B}(s) = \{B_i\}_{i\in I}$ which covers $A$. More precisely, each such set $B_i$ is a subset of $A$ with $\diam(B_i) \le c_A s$, and $\dist(B_i, B_j) \ge s$ whenever $i, j \in I, i\neq j$. 

For each $n \in \N$, define
\begin{equation}\label{eqn:dyadic_annulus}
        D_n := \left\{z : \tfrac{4}{c_A-1} c_A^{n+2} \le |z| < \tfrac{4}{c_A-1} c_A^{n+3}\right\}\subset \C,
\end{equation}
and set
$$
D_{\le n} := \bigcup_{1 \le i \le n} D_i
$$
and
$$
D_{\ge n} := \bigcup_{i \ge n} D_i.
$$
Finally, 
let $s_n := c_A^n$ and let $\mathcal{B}_n = \{B_{i,n}\}_{i\in I_n}$ denote a $c_A s_n$-bounded and $s_n$ disjoint cover of $A$ as described above.
        
Fix $0<\varepsilon<10^{-2}$. For a set $E \in \C$ and $\rho >0$, denote the open $\rho$–neighborhood by $N_\rho(E):= \{z\in\C : \exists w\in E \text{ so that } |z-w| \le \rho\}$. Define the (thickened) sets
        \[
            B_{i,n}^{\varepsilon}:=N_{\varepsilon s_n}(B_{i,n}),
            \qquad \mathcal B_n^{\varepsilon}:=\{B_{i,n}^{(\varepsilon)}\}_{i\in I_n},
        \]
        and set
        \[
            \Omega_n^\varepsilon := \left( \bigcup_{i\in I} B_{i,n}^\varepsilon\right)^c \subset \C.
        \]
See Figure \ref{fig:thickened-sets} for an illustration.

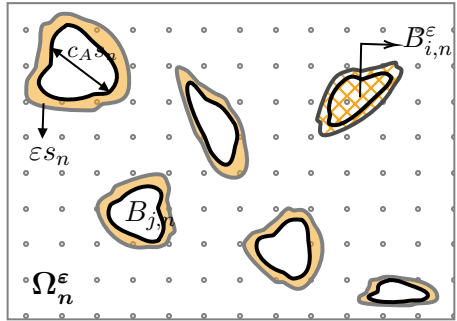
\begin{figure}[h]
    \centering    
    \usetikzlibrary{patterns}
    
     
    \tikzset{
    pattern size/.store in=\mcSize, 
    pattern size = 5pt,
    pattern thickness/.store in=\mcThickness, 
    pattern thickness = 0.3pt,
    pattern radius/.store in=\mcRadius, 
    pattern radius = 1pt}
    \makeatletter
    \pgfutil@ifundefined{pgf@pattern@name@_w08421ecj}{
    \makeatletter
    \pgfdeclarepatternformonly[\mcRadius,\mcThickness,\mcSize]{_w08421ecj}
    {\pgfpoint{-0.5*\mcSize}{-0.5*\mcSize}}
    {\pgfpoint{0.5*\mcSize}{0.5*\mcSize}}
    {\pgfpoint{\mcSize}{\mcSize}}
    {
    \pgfsetcolor{\tikz@pattern@color}
    \pgfsetlinewidth{\mcThickness}
    \pgfpathcircle\pgfpointorigin{\mcRadius}
    \pgfusepath{stroke}
    }}
    \makeatother
    
     
    \tikzset{
    pattern size/.store in=\mcSize, 
    pattern size = 5pt,
    pattern thickness/.store in=\mcThickness, 
    pattern thickness = 0.3pt,
    pattern radius/.store in=\mcRadius, 
    pattern radius = 1pt}
    \makeatletter
    \pgfutil@ifundefined{pgf@pattern@name@_9a8sws3w3}{
    \pgfdeclarepatternformonly[\mcThickness,\mcSize]{_9a8sws3w3}
    {\pgfqpoint{0pt}{0pt}}
    {\pgfpoint{\mcSize}{\mcSize}}
    {\pgfpoint{\mcSize}{\mcSize}}
    {
    \pgfsetcolor{\tikz@pattern@color}
    \pgfsetlinewidth{\mcThickness}
    \pgfpathmoveto{\pgfqpoint{0pt}{\mcSize}}
    \pgfpathlineto{\pgfpoint{\mcSize+\mcThickness}{-\mcThickness}}
    \pgfpathmoveto{\pgfqpoint{0pt}{0pt}}
    \pgfpathlineto{\pgfpoint{\mcSize+\mcThickness}{\mcSize+\mcThickness}}
    \pgfusepath{stroke}
    }}
    \makeatother
    
     
    \tikzset{
    pattern size/.store in=\mcSize, 
    pattern size = 5pt,
    pattern thickness/.store in=\mcThickness, 
    pattern thickness = 0.3pt,
    pattern radius/.store in=\mcRadius, 
    pattern radius = 1pt}
    \makeatletter
    \pgfutil@ifundefined{pgf@pattern@name@_sxav00h5x}{
    \pgfdeclarepatternformonly[\mcThickness,\mcSize]{_sxav00h5x}
    {\pgfqpoint{0pt}{0pt}}
    {\pgfpoint{\mcSize}{\mcSize}}
    {\pgfpoint{\mcSize}{\mcSize}}
    {
    \pgfsetcolor{\tikz@pattern@color}
    \pgfsetlinewidth{\mcThickness}
    \pgfpathmoveto{\pgfqpoint{0pt}{\mcSize}}
    \pgfpathlineto{\pgfpoint{\mcSize+\mcThickness}{-\mcThickness}}
    \pgfpathmoveto{\pgfqpoint{0pt}{0pt}}
    \pgfpathlineto{\pgfpoint{\mcSize+\mcThickness}{\mcSize+\mcThickness}}
    \pgfusepath{stroke}
    }}
    \makeatother
    \tikzset{every picture/.style={line width=0.75pt}} 
    
    \begin{tikzpicture}[x=0.75pt,y=0.75pt,yscale=-.6,xscale=.6]
    
    \draw  [color={rgb, 255:red, 128; green, 128; blue, 128 }  ,draw opacity=1 ][pattern=_w08421ecj,pattern size=13.5pt,pattern thickness=0.75pt,pattern radius=0.75pt, pattern color={rgb, 255:red, 128; green, 128; blue, 128}] (176.31,76.8) -- (551.81,76.8) -- (551.81,342.8) -- (176.31,342.8) -- cycle ;
    \draw [color={rgb, 255:red, 128; green, 128; blue, 128 }  ,draw opacity=1 ][fill={rgb, 255:red, 245; green, 166; blue, 35 }  ,fill opacity=0.54 ][line width=1.5] [line join = round][line cap = round]   (203.93,105.29) .. controls (203.93,115.13) and (201.26,120.17) .. (200.11,130.13) .. controls (199.78,132.93) and (192.7,139.75) .. (192.47,143.01) .. controls (192.13,147.93) and (190.57,154.98) .. (194.38,158.65) .. controls (202.86,166.83) and (238.12,167.24) .. (247.83,166.93) .. controls (251.54,166.81) and (253.69,161.77) .. (257.38,161.41) .. controls (268.57,160.33) and (277.5,159.24) .. (278.38,145.77) .. controls (279.97,121.2) and (266.07,96.81) .. (244.01,91.49) .. controls (239.05,90.3) and (232.74,86.46) .. (228.74,85.97) .. controls (226.22,85.67) and (223.6,85.49) .. (221.11,85.97) .. controls (216.4,86.88) and (213.58,92.69) .. (211.56,93.33) .. controls (209.52,93.99) and (207.17,93.89) .. (205.83,95.17) .. controls (204.4,96.56) and (203.93,101.99) .. (203.93,104.37) ;
    \draw [color={rgb, 255:red, 0; green, 0; blue, 0 }  ,draw opacity=1 ][fill={rgb, 255:red, 255; green, 255; blue, 255 }  ,fill opacity=1 ][line width=1.5] [line join = round][line cap = round]   (224.92,96.09) .. controls (210.03,105.66) and (214.21,111.98) .. (212.52,128.29) .. controls (212.5,128.41) and (212.47,128.53) .. (212.4,128.67) .. controls (211.14,131.42) and (200.15,138.99) .. (202.97,145.77) .. controls (209.31,161.05) and (235.65,159.46) .. (246.88,154.05) .. controls (250.9,152.11) and (258.86,154.46) .. (262.15,151.29) .. controls (271.47,142.31) and (268.71,132.57) .. (262.55,124.11) .. controls (258.88,119.08) and (254,114.51) .. (249.74,110.81) .. controls (241.78,103.91) and (237.5,95.17) .. (225.88,95.17) ;
    \draw [color={rgb, 255:red, 128; green, 128; blue, 128 }  ,draw opacity=1 ][fill={rgb, 255:red, 245; green, 166; blue, 35 }  ,fill opacity=0.54 ][line width=1.5] [line join = round][line cap = round]   (314.95,261.31) .. controls (312.13,254.44) and (312.51,250.18) .. (310.44,242.91) .. controls (309.85,240.87) and (312.72,234.14) .. (311.94,231.8) .. controls (310.76,228.28) and (309.8,222.92) .. (306.17,221.42) .. controls (298.06,218.08) and (273.98,227.62) .. (267.47,230.54) .. controls (264.99,231.66) and (264.97,235.78) .. (262.57,237.05) .. controls (255.27,240.93) and (249.51,244.18) .. (252.77,253.82) .. controls (258.72,271.4) and (275.15,284.55) .. (291.65,282.11) .. controls (295.37,281.56) and (300.76,282.48) .. (303.61,281.7) .. controls (305.41,281.21) and (307.24,280.61) .. (308.8,279.58) .. controls (311.74,277.63) and (312,272.79) .. (313.18,271.78) .. controls (314.38,270.76) and (316.01,270.17) .. (316.55,268.9) .. controls (317.13,267.54) and (315.89,263.62) .. (315.21,261.95) ;
    \draw [color={rgb, 255:red, 128; green, 128; blue, 128 }  ,draw opacity=1 ][fill={rgb, 255:red, 245; green, 166; blue, 35 }  ,fill opacity=0.54 ][line width=1.5] [line join = round][line cap = round]   (353.21,153.78) .. controls (347.49,151.57) and (343.57,147.69) .. (337.34,144.27) .. controls (335.6,143.31) and (329,134.49) .. (327.01,133.52) .. controls (324.03,132.07) and (319.34,128.88) .. (318.62,131.97) .. controls (317.02,138.84) and (329.89,174.99) .. (333.68,185.04) .. controls (335.13,188.88) and (338.86,192.22) .. (340.44,196.09) .. controls (345.23,207.84) and (349.18,217.26) .. (357.35,221.19) .. controls (372.23,228.36) and (381.25,219.56) .. (376.15,198.09) .. controls (375,193.25) and (374.89,187.64) .. (373.68,183.64) .. controls (372.92,181.11) and (372.05,178.46) .. (370.84,175.79) .. controls (368.56,170.74) and (364.14,166.54) .. (363.01,164.32) .. controls (361.87,162.08) and (361.06,159.69) .. (359.81,158.02) .. controls (358.47,156.24) and (355.14,154.53) .. (353.75,153.99) ;
    \draw [color={rgb, 255:red, 0; green, 0; blue, 0 }  ,draw opacity=1 ][fill={rgb, 255:red, 255; green, 255; blue, 255 }  ,fill opacity=1 ][line width=1.5] [line join = round][line cap = round]   (366.37,177.43) .. controls (355.27,159.98) and (353.15,162.85) .. (343.03,157.44) .. controls (342.96,157.4) and (342.87,157.33) .. (342.77,157.24) .. controls (340.7,155.32) and (332.21,142.33) .. (329.31,143.69) .. controls (322.78,146.77) and (333.5,174.2) .. (340.82,186.96) .. controls (343.44,191.52) and (345.04,199.18) .. (348.1,203.27) .. controls (356.79,214.88) and (361.43,214.23) .. (364.06,209.8) .. controls (365.62,207.16) and (366.47,203.17) .. (367.04,199.63) .. controls (368.09,193) and (371.58,190.56) .. (367.26,178.62) ;
    \draw [color={rgb, 255:red, 128; green, 128; blue, 128 }  ,draw opacity=1 ][fill={rgb, 255:red, 245; green, 166; blue, 35 }  ,fill opacity=0.54 ][line width=1.5] [line join = round][line cap = round]   (406.26,254.43) .. controls (401.5,260.12) and (397.56,261.78) .. (392.09,267) .. controls (390.56,268.46) and (383.26,269.06) .. (381.56,270.84) .. controls (378.98,273.53) and (374.69,276.87) .. (375.06,280.79) .. controls (375.88,289.51) and (395.54,306.38) .. (401.16,310.77) .. controls (403.31,312.45) and (406.96,310.55) .. (409.21,312.08) .. controls (416.04,316.73) and (421.59,320.32) .. (428.61,312.94) .. controls (441.39,299.48) and (445.37,278.82) .. (435.52,265.35) .. controls (433.3,262.32) and (431.61,257.12) .. (429.59,254.96) .. controls (428.32,253.59) and (426.93,252.26) .. (425.29,251.36) .. controls (422.2,249.67) and (417.8,251.69) .. (416.35,251.12) .. controls (414.89,250.53) and (413.61,249.37) .. (412.24,249.48) .. controls (410.76,249.6) and (407.86,252.52) .. (406.71,253.9) ;
    \draw [color={rgb, 255:red, 0; green, 0; blue, 0 }  ,draw opacity=1 ][fill={rgb, 255:red, 255; green, 255; blue, 255 }  ,fill opacity=1 ][line width=1.5] [line join = round][line cap = round]   (422.54,259.01) .. controls (409.53,257.52) and (408.82,263.14) .. (399.97,271.78) .. controls (399.91,271.84) and (399.83,271.9) .. (399.72,271.94) .. controls (397.68,272.94) and (387.83,272.14) .. (386.13,277.39) .. controls (382.31,289.21) and (397.92,300.71) .. (406.86,302.88) .. controls (410.06,303.65) and (413.41,308.76) .. (416.79,308.48) .. controls (426.39,307.68) and (429.55,300.75) .. (430.17,292.95) .. controls (430.54,288.31) and (430.01,283.37) .. (429.4,279.22) .. controls (428.26,271.48) and (430.07,264.41) .. (423.53,258.93) ;
    \draw [color={rgb, 255:red, 74; green, 74; blue, 74 }  ,draw opacity=1 ][pattern=_9a8sws3w3,pattern size=6pt,pattern thickness=0.75pt,pattern radius=0pt, pattern color={rgb, 255:red, 245; green, 166; blue, 35}][line width=1.5] [line join = round][line cap = round]   (494.57,128.69) .. controls (486.51,131.6) and (482.83,131.44) .. (474.87,133.67) .. controls (472.64,134.3) and (468.27,131.92) .. (465.64,132.74) .. controls (461.67,133.99) and (456.16,135.11) .. (452.5,138.56) .. controls (444.34,146.24) and (437.93,168.26) .. (436.51,174.2) .. controls (435.97,176.47) and (439.74,176.31) .. (439.39,178.5) .. controls (438.35,185.13) and (437.71,190.35) .. (448.59,186.91) .. controls (468.44,180.63) and (490.82,164.77) .. (498.97,149.5) .. controls (500.81,146.07) and (505.04,141.01) .. (506.12,138.39) .. controls (506.81,136.73) and (507.4,135.05) .. (507.44,133.64) .. controls (507.5,130.99) and (503.23,130.95) .. (503.05,129.89) .. controls (502.86,128.82) and (503.35,127.33) .. (502.53,126.88) .. controls (501.64,126.4) and (497.27,127.71) .. (495.32,128.42) ;
    \draw [color={rgb, 255:red, 0; green, 0; blue, 0 }  ,draw opacity=1 ][pattern=_sxav00h5x,pattern size=6pt,pattern thickness=0.75pt,pattern radius=0pt, pattern color={rgb, 255:red, 245; green, 166; blue, 35}][line width=1.5] [line join = round][line cap = round]   (498.49,139.01) .. controls (493.21,132.59) and (487.32,137.06) .. (474.25,140.83) .. controls (474.15,140.86) and (474.06,140.87) .. (473.96,140.87) .. controls (471.92,140.9) and (467.61,136.32) .. (461.57,140.08) .. controls (447.96,148.54) and (444.73,164.43) .. (447.23,169.8) .. controls (448.12,171.72) and (444.83,177.36) .. (446.86,178.46) .. controls (452.62,181.59) and (461.07,177) .. (469.06,170.67) .. controls (473.81,166.9) and (478.4,162.52) .. (482.16,158.78) .. controls (489.18,151.79) and (497.08,146.55) .. (499.08,139.33) ;
    \draw [color={rgb, 255:red, 128; green, 128; blue, 128 }  ,draw opacity=1 ][fill={rgb, 255:red, 245; green, 166; blue, 35 }  ,fill opacity=0.54 ][line width=1.5] [line join = round][line cap = round]   (480.69,312.79) .. controls (480.69,315.69) and (478.65,317.18) .. (477.78,320.12) .. controls (477.53,320.95) and (472.13,322.96) .. (471.96,323.92) .. controls (471.7,325.37) and (470.51,327.45) .. (473.41,328.54) .. controls (479.88,330.95) and (506.74,331.07) .. (514.14,330.98) .. controls (516.96,330.94) and (518.59,329.46) .. (521.41,329.35) .. controls (529.93,329.03) and (536.74,328.71) .. (537.41,324.74) .. controls (538.62,317.48) and (528.03,310.28) .. (511.23,308.72) .. controls (507.44,308.36) and (502.64,307.23) .. (499.59,307.09) .. controls (497.67,307) and (495.68,306.95) .. (493.77,307.09) .. controls (490.19,307.36) and (488.04,309.07) .. (486.5,309.26) .. controls (484.95,309.45) and (483.16,309.42) .. (482.14,309.8) .. controls (481.05,310.21) and (480.69,311.81) .. (480.69,312.52) ;
    \draw [color={rgb, 255:red, 0; green, 0; blue, 0 }  ,draw opacity=1 ][fill={rgb, 255:red, 255; green, 255; blue, 255 }  ,fill opacity=1 ][line width=1.5] [line join = round][line cap = round]   (496.68,310.07) .. controls (485.34,312.9) and (488.52,314.76) .. (487.23,319.58) .. controls (487.22,319.61) and (487.19,319.65) .. (487.14,319.69) .. controls (486.18,320.5) and (477.81,322.73) .. (479.96,324.74) .. controls (484.79,329.24) and (504.86,328.77) .. (513.41,327.18) .. controls (516.47,326.61) and (522.54,327.3) .. (525.04,326.36) .. controls (532.15,323.71) and (530.04,320.84) .. (525.35,318.34) .. controls (522.55,316.86) and (518.83,315.51) .. (515.59,314.42) .. controls (509.53,312.38) and (506.26,309.8) .. (497.41,309.8) ;
    \draw    (215.54,116.1) -- (259.76,149.48) ;
    \draw [shift={(262.15,151.29)}, rotate = 217.06] [fill={rgb, 255:red, 0; green, 0; blue, 0 }  ][line width=0.08]  [draw opacity=0] (8.93,-4.29) -- (0,0) -- (8.93,4.29) -- cycle    ;
    \draw [shift={(213.15,114.29)}, rotate = 37.06] [fill={rgb, 255:red, 0; green, 0; blue, 0 }  ][line width=0.08]  [draw opacity=0] (8.93,-4.29) -- (0,0) -- (8.93,4.29) -- cycle    ;
    \draw    (472.5,156) -- (473.5,111) -- (501.5,111.93) ;
    \draw [shift={(503.5,112)}, rotate = 181.91] [color={rgb, 255:red, 0; green, 0; blue, 0 }  ][line width=0.75]    (10.93,-3.29) .. controls (6.95,-1.4) and (3.31,-0.3) .. (0,0) .. controls (3.31,0.3) and (6.95,1.4) .. (10.93,3.29)   ;
    \draw [color={rgb, 255:red, 0; green, 0; blue, 0 }  ,draw opacity=1 ][fill={rgb, 255:red, 255; green, 255; blue, 255 }  ,fill opacity=1 ][line width=1.5] [line join = round][line cap = round]   (303.52,275.48) .. controls (310.9,264.66) and (306.26,261.41) .. (302.74,249.56) .. controls (302.71,249.48) and (302.71,249.38) .. (302.71,249.26) .. controls (302.78,246.99) and (308.08,238.65) .. (304.22,234.7) .. controls (295.54,225.81) and (278.1,234.26) .. (272.02,241.16) .. controls (269.84,243.63) and (263.76,244.22) .. (262.43,247.35) .. controls (258.66,256.21) and (263.32,262.24) .. (269.94,266.42) .. controls (273.87,268.9) and (278.49,270.74) .. (282.44,272.13) .. controls (289.83,274.73) and (295.24,279.63) .. (303.13,276.39) ;
    \draw    (206.5,161) -- (205.6,188) ;
    \draw [shift={(205.5,191)}, rotate = 271.91] [fill={rgb, 255:red, 0; green, 0; blue, 0 }  ][line width=0.08]  [draw opacity=0] (8.93,-4.29) -- (0,0) -- (8.93,4.29) -- cycle    ;
    
    \draw (224,111) node [anchor=north west][inner sep=0.75pt]  [font=\footnotesize]  {$c_{A} s_{n}$};
    \draw (194,301) node [anchor=north west][inner sep=0.75pt]  [font=\large]  {$\boldsymbol{\Omega _{n}^{\varepsilon }}$};
    \draw (272.02,241.16) node [anchor=north west][inner sep=0.75pt]    {$B_{j,n}$};
    \draw (505.81,94.26) node [anchor=north west][inner sep=0.75pt]    {$B_{i,n}^{\varepsilon }$};
    \draw (192,196) node [anchor=north west][inner sep=0.75pt]    {$\varepsilon s_{n}$};    
    \end{tikzpicture}
\caption{$\varepsilon s_n$ thickening of the elements $B_{i,n}$}
\label{fig:thickened-sets}
\end{figure}

We note that
$$
\bigcup_{i\in I_n} B_{i,n}^\varepsilon \subset \bigcup_{i\in I_{n+1}} B_{i,{n+1}}^\varepsilon
$$
and
$$
\left(\bigcup_{i\in I_{n+1}} B_{i,{n+1}}^\varepsilon\right)^c = \Omega_{n+1}^\varepsilon \subset \Omega_n^\varepsilon = \left(\bigcup_{i\in I_n} B_{i,n}^\varepsilon\right)^c.
$$
For fixed $n$, the set
\begin{equation*}
            \Omega_{n}^\varepsilon \cap D_n \cap \R = \bigcup \{ I : I \in \widetilde\cI_n \}
\end{equation*}
is a finite union of closed intervals. Indeed, the set $A\cap D_n$ is finite, hence $\{B_{i,n}^\varepsilon: B_{i,n}^{\varepsilon} \cap D_n \ne \emptyset \}$ is a finite union of open balls and $\bigcup_{i \in I_n} B_{i,n}^\varepsilon \cap D_n \cap \R$ is a finite union of open intervals. Moreover, since the width of $D_n$ is $4 c_A^{n+2} = 4c_A^2 s_n$ and there exists an interval $I \in \widetilde\cI_n$ of length at least $s_n(1-2\varepsilon)$, this union is nonempty.
For each $x\in \Omega_{n}^\varepsilon \cap D_n \cap \R$ and $\alpha \in A$, we have
\begin{equation}\label{eqn:min_dist_z_k_alpha}
                |\alpha - x| \ge \varepsilon s_n        
\end{equation}
We now define an admissible subfamily $\cI_n$ of $\widetilde\cI_n$:
        \begin{equation*}
\cI_n := \left\{ I \in \widetilde\cI_n : {\mbox{$\forall\, B^\varepsilon \in \mathcal{B}_n^\varepsilon$, $B^\varepsilon$ is disjoint from the} \atop \mbox{closure of one of the components of $I^c$}} \right\}.
        \end{equation*}
For each $I \in \cI_n$, we have $\ell(I) \ge  s_n -2\varepsilon s_n$.

See Figure \ref{fig:In} for an illustratation of the family $\cI_n$.

\begin{figure}[h]
    \centering
    
    \tikzset{every picture/.style={line width=0.75pt}} 
    
    \begin{tikzpicture}[x=0.75pt,y=0.75pt,yscale=-.6,xscale=.6]
    
    \draw  [dash pattern={on 4.5pt off 4.5pt}]  (47.5,578.88) -- (524.5,577.88) ;
    \draw [shift={(527.5,577.88)}, rotate = 179.88] [fill={rgb, 255:red, 0; green, 0; blue, 0 }  ][line width=0.08]  [draw opacity=0] (8.93,-4.29) -- (0,0) -- (8.93,4.29) -- cycle    ;
    \draw  [draw opacity=0] (75.79,510.63) .. controls (102.48,521.75) and (121.25,548.12) .. (121.25,578.88) .. controls (121.25,608.77) and (103.52,634.52) .. (78.02,646.15) -- (47.5,578.88) -- cycle ; \draw   (75.79,510.63) .. controls (102.48,521.75) and (121.25,548.12) .. (121.25,578.88) .. controls (121.25,608.77) and (103.52,634.52) .. (78.02,646.15) ;  
    \draw  [draw opacity=0] (333.01,443.49) .. controls (395.33,479.28) and (433.31,526.82) .. (433.31,579) .. controls (433.31,627.23) and (400.85,671.51) .. (346.72,706.19) -- (47.5,579) -- cycle ; \draw   (333.01,443.49) .. controls (395.33,479.28) and (433.31,526.82) .. (433.31,579) .. controls (433.31,627.23) and (400.85,671.51) .. (346.72,706.19) ;  
    \draw [color={rgb, 255:red, 128; green, 128; blue, 128 }  ,draw opacity=1 ][fill={rgb, 255:red, 245; green, 166; blue, 35 }  ,fill opacity=0.23 ][line width=0.75] [line join = round][line cap = round]   (169.09,581.72) .. controls (164.59,580.19) and (162.64,578.19) .. (158.24,576.11) .. controls (157,575.53) and (154.82,571.24) .. (153.35,570.63) .. controls (151.15,569.7) and (148.13,567.9) .. (145.95,569.06) .. controls (141.09,571.65) and (136.24,587.66) .. (135.1,592.14) .. controls (134.67,593.85) and (136.69,595.61) .. (136.37,597.35) .. controls (135.38,602.62) and (134.7,606.87) .. (140.75,609.36) .. controls (151.78,613.92) and (164.77,611.38) .. (170.12,602.15) .. controls (171.32,600.07) and (173.91,597.79) .. (174.66,596.05) .. controls (175.13,594.94) and (175.56,593.78) .. (175.67,592.57) .. controls (175.87,590.28) and (173.59,588.09) .. (173.56,587.07) .. controls (173.53,586.04) and (173.89,584.98) .. (173.47,584.17) .. controls (173.03,583.3) and (170.61,582.24) .. (169.52,581.87) ;
    \draw [color={rgb, 255:red, 128; green, 128; blue, 128 }  ,draw opacity=1 ][fill={rgb, 255:red, 255; green, 255; blue, 255 }  ,fill opacity=1 ][line width=0.75] [line join = round][line cap = round]   (170.53,592.73) .. controls (168.12,584.45) and (164.68,585.37) .. (157.44,582.06) .. controls (157.39,582.03) and (157.34,582) .. (157.28,581.95) .. controls (156.19,580.94) and (154.18,574.76) .. (150.7,574.98) .. controls (142.88,575.49) and (140.13,587.75) .. (141.12,593.71) .. controls (141.47,595.84) and (139.35,599.11) .. (140.37,601.1) .. controls (143.25,606.75) and (148.06,607.01) .. (152.75,605.52) .. controls (155.53,604.63) and (158.27,603.12) .. (160.52,601.75) .. controls (164.73,599.2) and (169.29,598.61) .. (170.83,593.31) ;
    \draw [color={rgb, 255:red, 128; green, 128; blue, 128 }  ,draw opacity=1 ][fill={rgb, 255:red, 245; green, 166; blue, 35 }  ,fill opacity=0.24 ][line width=0.75] [line join = round][line cap = round]   (363.6,545.59) .. controls (360.47,549.16) and (357.87,550.2) .. (354.27,553.48) .. controls (353.26,554.4) and (348.46,554.78) .. (347.33,555.9) .. controls (345.64,557.58) and (342.81,559.68) .. (343.05,562.14) .. controls (343.59,567.63) and (356.54,578.22) .. (360.24,580.98) .. controls (361.65,582.04) and (364.06,580.84) .. (365.54,581.8) .. controls (370.04,584.73) and (373.7,586.98) .. (378.31,582.34) .. controls (386.74,573.89) and (389.35,560.91) .. (382.87,552.45) .. controls (381.41,550.54) and (380.29,547.28) .. (378.96,545.92) .. controls (378.12,545.06) and (377.21,544.22) .. (376.13,543.66) .. controls (374.1,542.59) and (371.2,543.87) .. (370.25,543.51) .. controls (369.28,543.14) and (368.44,542.41) .. (367.53,542.48) .. controls (366.56,542.56) and (364.66,544.39) .. (363.89,545.25) ;
    \draw [color={rgb, 255:red, 128; green, 128; blue, 128 }  ,draw opacity=1 ][fill={rgb, 255:red, 255; green, 255; blue, 255 }  ,fill opacity=1 ][line width=0.75] [line join = round][line cap = round]   (374.32,548.47) .. controls (365.75,547.53) and (365.29,551.06) .. (359.46,556.49) .. controls (359.42,556.52) and (359.36,556.56) .. (359.3,556.59) .. controls (357.95,557.21) and (351.46,556.71) .. (350.35,560.01) .. controls (347.83,567.44) and (358.11,574.66) .. (363.99,576.02) .. controls (366.1,576.51) and (368.31,579.72) .. (370.54,579.54) .. controls (376.86,579.04) and (378.94,574.69) .. (379.35,569.79) .. controls (379.59,566.87) and (379.24,563.77) .. (378.84,561.16) .. controls (378.08,556.3) and (379.28,551.86) .. (374.97,548.41) ;
    \draw [color={rgb, 255:red, 128; green, 128; blue, 128 }  ,draw opacity=1 ][fill={rgb, 255:red, 245; green, 166; blue, 35 }  ,fill opacity=0.25 ][line width=0.75] [line join = round][line cap = round]   (326.11,614) .. controls (320.25,616.9) and (317.2,617.75) .. (311.25,620.41) .. controls (309.58,621.16) and (305.38,621.46) .. (303.43,622.37) .. controls (300.5,623.74) and (296.27,625.45) .. (294.16,627.44) .. controls (289.46,631.89) and (289.89,640.5) .. (290.26,642.74) .. controls (290.4,643.6) and (293.45,642.62) .. (293.73,643.41) .. controls (294.59,645.78) and (295.41,647.61) .. (303.44,643.84) .. controls (318.1,636.98) and (332.35,626.44) .. (335.09,619.57) .. controls (335.71,618.02) and (337.88,615.37) .. (338.09,614.27) .. controls (338.22,613.57) and (338.27,612.89) .. (337.94,612.44) .. controls (337.31,611.57) and (333.8,612.6) .. (333.38,612.31) .. controls (332.95,612.01) and (332.96,611.42) .. (332.17,611.48) .. controls (331.32,611.54) and (328.08,613.03) .. (326.66,613.73) ;
    \draw [color={rgb, 255:red, 128; green, 128; blue, 128 }  ,draw opacity=1 ][fill={rgb, 255:red, 255; green, 255; blue, 255 }  ,fill opacity=1 ][line width=0.75] [line join = round][line cap = round]   (331.99,616.34) .. controls (326.01,615.58) and (322.33,618.45) .. (312.58,622.85) .. controls (312.51,622.88) and (312.44,622.91) .. (312.36,622.93) .. controls (310.69,623.44) and (305.98,623.03) .. (301.99,625.71) .. controls (293.02,631.74) and (294.47,637.61) .. (297.91,638.71) .. controls (299.14,639.11) and (297.89,641.71) .. (299.84,641.57) .. controls (305.37,641.16) and (311.12,637.63) .. (316.03,633.65) .. controls (318.96,631.28) and (321.59,628.76) .. (323.7,626.65) .. controls (327.66,622.7) and (332.78,619.09) .. (332.55,616.3) ;
    \draw [color={rgb, 255:red, 155; green, 155; blue, 155 }  ,draw opacity=1 ][fill={rgb, 255:red, 245; green, 166; blue, 35 }  ,fill opacity=0.27 ][line width=1.5] [line join = round][line cap = round]   (189.5,565) .. controls (189.5,558.19) and (197.08,550.56) .. (200.5,546) .. controls (217.61,523.18) and (223.5,521.63) .. (254.5,520) .. controls (259.56,519.73) and (219.09,522.67) .. (262.54,519.57) .. controls (264.93,519.4) and (267.58,519.21) .. (270.5,519) .. controls (274.17,519.24) and (280.5,520) .. (280.5,523) .. controls (285.49,523.83) and (291.5,522) .. (294.5,528) .. controls (295.99,528.75) and (295.92,531.47) .. (297.5,532) .. controls (313.84,537.45) and (319.99,561.07) .. (317.5,581) .. controls (316.22,591.24) and (294.77,589.17) .. (289.5,588) .. controls (286.18,587.26) and (282.5,579) .. (265.5,580) .. controls (276.5,575) and (285.5,566) .. (270.5,570) .. controls (274.11,565.99) and (271.6,563.42) .. (266.6,562.12) .. controls (259.12,560.16) and (246.09,561.01) .. (239.5,564) .. controls (241.5,555) and (230.88,562.72) .. (230.5,563) .. controls (229.7,563.6) and (233.5,564) .. (230.5,566) .. controls (229.62,568.2) and (223.61,574.89) .. (222.5,576) .. controls (221.22,577.28) and (218.62,577.29) .. (216.5,578) .. controls (215.76,578.25) and (211.39,581.93) .. (210.5,582) .. controls (193.88,583.38) and (190.5,576.64) .. (190.5,560) ;
    \draw  [color={rgb, 255:red, 128; green, 128; blue, 128 }  ,draw opacity=1 ][fill={rgb, 255:red, 255; green, 255; blue, 255 }  ,fill opacity=1 ][line width=1.5]  (197.5,570) .. controls (187.5,560) and (220.9,534.64) .. (242.5,528) .. controls (264.1,521.36) and (305.33,531.45) .. (308.5,560) .. controls (311.67,588.55) and (296.23,579.27) .. (283.5,578) .. controls (270.77,576.73) and (298.5,571) .. (284.5,573) .. controls (270.5,575) and (279.51,565.01) .. (276.5,562) .. controls (273.49,558.99) and (255.28,553.38) .. (241.5,552) .. controls (227.72,550.62) and (224.95,558.26) .. (218.5,566) .. controls (212.05,573.74) and (216.44,571.61) .. (214.5,572) .. controls (212.56,572.39) and (198.5,575.19) .. (198.5,568) .. controls (198.5,560.81) and (207.5,580) .. (197.5,570) -- cycle ;
    \draw [color={springgreen}  ,draw opacity=1 ][line width=4]    (122.5,578) -- (138.5,578) ;
    \draw [color={springgreen}  ,draw opacity=1 ][line width=4]    (166.5,578) -- (193.5,578) ;
    \draw [color={springgreen}  ,draw opacity=1 ][line width=4]    (320.5,578) -- (355.5,578) ;
    \draw [color={springgreen}  ,draw opacity=1 ][line width=4]    (384.5,578) -- (433.5,578) ;
    \draw [color={rgb, 255:red, 183; green, 13; blue, 34 }  ,draw opacity=1 ][line width=4]    (222.5,578) -- (264.5,578) ;
    
    \draw (272,539) node [anchor=north west][inner sep=0.75pt]    {$B_{i,n}^{\varepsilon }$};
    \draw (106,490) node [anchor=north west][inner sep=0.75pt]  [font=\large]  {$D_{n}$};
    
    \end{tikzpicture}
    \caption{Elements of an admissible subfamily $\mathcal{I}_n$ shown in green.}
    \label{fig:In}
\end{figure}
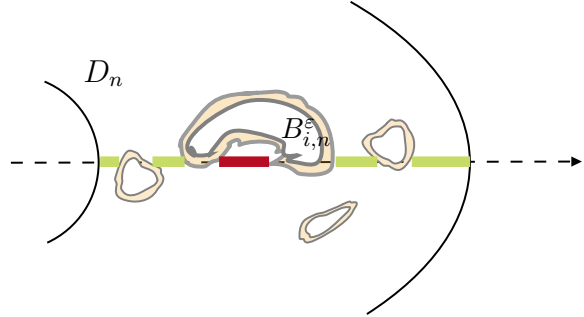

We now assume that the zeroset $A$ has been enumerated, $A = \{\alpha_1,\alpha_2,\alpha_3,\ldots\}$ so that $0<|\alpha_1| \le |\alpha_2| \le |\alpha_3| \le \cdots$.

\begin{lemma}\label{lemma:lower_bound_annulus_dueto_alphaj}
For $A$, $c_A$, $\varepsilon$, $a_d$, $d$ as above, there exists $J_A\in\N$ such that
\begin{equation}\label{eqn:tail_infinite_sum}
                    \sum_{j>J_A+1}^{\infty}\frac{1}{|\alpha_j|}\ \le \frac{\varepsilon d a_d}{16} \cdot \frac{c_A - 1}{c_A^2(c_A^2 + \varepsilon)}.
\end{equation}
Moreover, there exists $N_A\in\N$ such that if $\alpha_j \in D_{\ge N_A}$ then $j>J_A + 1$.
\end{lemma}
    
\begin{proof}
Since $\sum_j |\alpha_j|^{-1}$ converges, there exists $J_A$ so that \eqref{eqn:tail_infinite_sum} holds. Due to how the values $\alpha_j$ are enumerated, there exists $N_A$ so that $j>J_A + 1$ whenever $\alpha_j \in D_{\ge N_A}$.
\end{proof}
    
\paragraph{\bf Construction of the set $X$.}
Fix 
\begin{equation}\label{eq:choice-of-delta}
0 < \delta < c_A.
\end{equation}
Let $R_0$ be as in Lemma \ref{lemma:radial-growth} and $N_A$ be as in Lemma \ref{lemma:lower_bound_annulus_dueto_alphaj}. Fix $n_q \ge 1$ so that $R_0 \in D_{n_q}$. We first construct a set $\tilde X \subset \C$ as follows:
\begin{equation*}
\tilde X \subset \bigcup_{n > \max\{n_q, N_A\}} \bigcup_{I \in \mathcal{I}_n}  I \subset \Omega_n^\varepsilon
\end{equation*}
and
\begin{equation*}
\tilde X \cap D_n := \left\{ (\inf I) + \varepsilon s_n + m\delta : I \in \mathcal{I}_n, m = 1, \ldots, \left\lfloor \tfrac{\ell(I)-\varepsilon s_n}{\delta} \right\rfloor\right\}
\end{equation*}
for $n<\max\{n_q,N_A\}$.

Enumerate $\tilde X = \{z_1<z_2<z_3<\cdots\}$ in increasing order.

\begin{proposition}\label{prop:ratio_consecutive_terms_tildeX}
There exists $K \in \N$ and $\Lambda>1$ so that for all $k \ge K$, we have
$$
\left|\frac{f(z_{k+1})}{f(z_k)}\right| \ge \Lambda\,.
$$
\end{proposition}

We defer the proof of Proposition \ref{prop:ratio_consecutive_terms_tildeX} momentarily. We define
\begin{equation}\label{eq:X}
X = \tilde{X} \cap D_{\ge K}
\end{equation}
and show that this set $X$ verifies the conclusion in Theorem~\ref{thm:entire_genus_1_decrease_ndim}.

\begin{proof}[Proof of Theorem~\ref{thm:entire_genus_1_decrease_ndim}]

Let $X$ be as in \eqref{eq:X}, where $K$ is as in Proposition \ref{prop:ratio_consecutive_terms_tildeX}. By Corollary \ref{cor:growing_arithmetic_patches}, $\dim_N(X) = 1$. On the other hand, Proposition \ref{prop:ratio_consecutive_terms_tildeX} and Lemma \ref{lemma:ratio>1_ndim_0} together imply that $\dim_N(f(X)) = 0$. This completes the proof.
\end{proof}

It remains to prove Proposition \ref{prop:ratio_consecutive_terms_tildeX}. First, we observe that exactly one of the following holds true for each pair of consecutive terms $z_k$, $z_{k+1}$ selected from $\tilde X$:
\begin{itemize}
\item If $z_k$ and $z_{k+1}$ lie in a common interval $I \in \cI_n$, then
\begin{equation}\label{case-1}
 |z_{k+1}-z_k|=\delta;
\end{equation}
\item If $z_k$ and $z_{k+1}$ lie in distinct intervals in the same annulus $D_n$, then
\begin{equation}\label{case-2}
 \varepsilon s_n \le |z_{k+1}-z_k|\le (c_A+2\varepsilon)s_n;
\end{equation}
\item If $z_k \in D_n$ and $z_{k+1} \in D_{n+1}$, then
\begin{equation}\label{case-3}
\varepsilon s_n \le |z_{k+1}-z_k|\le (2c_A^2+2\varepsilon)s_n.
\end{equation}
\end{itemize}

For fixed $k \in \N$, let $n(k) \in \N$ be such that $z_k \in D_{n(k)}$. Observe that
\begin{equation*}
            \left| \frac{f(z_{k+1})}{f(z_k)} \right| = \left|\frac{z_{k+1}^m}{z_k^m} \right| \cdot\left| \frac{e^{q(z_{k+1})}}{e^{q(z_k)}} \right| \cdot\left| \frac{P(z_{k+1})}{P(z_k)} \right| \ge E(z_k) Q(z_k)
\end{equation*}
where $E(z_k) : = e^{\Re(q({z_{k+1}}) - {q(z_k)})}$ and $Q(z_k) := \left| \frac{P(z_{k+1})}{P(z_k)}\right|$.

We estimate
\begin{equation}\begin{split}\label{eq:Q-z-k}
        Q(z_k) = \left| \frac{P(z_{k+1})}{P(z_k)}\right| 
&= \prod_{j=1}^{\infty} \left| \frac{\alpha_j - z_{k+1}}{\alpha_j - z_k} \right| \\
&\ge
    \underbrace{\prod_{j\le J(k)+1}
    \Bigl|\tfrac{\alpha_j - z_{k+1}}{\alpha_j - z_k}\Bigr|}_{\textstyle Q_1(k)}
    \cdot
    \underbrace{\prod_{j>J(k)+1}
    \Bigl(1 - \tfrac{|z_{k+1}-z_k|}{|\alpha_j-z_k|}\Bigr)}_{\textstyle Q_2(k)}.
\end{split}\end{equation}
where
\begin{equation}\label{eq:Jk}
J(k) := \inf\bigl\{j\in\N:\, \alpha_j\in D_{\ge n(k)+2}\bigr\}.
\end{equation}
Note that $J(k)$ is finite since $\sum_j|\alpha_j|^{-1}<\infty$. Moreover,  by Lemma \ref{lemma:sublinear_growth_zeroes}
    \begin{equation}\label{eqn:J(k)_ratio_cn}
             \lim_{k\to\infty}\frac{J(k)}{s_{n(k)}}=0.
    \end{equation}
We now prove that
\begin{equation}\label{eq:alpha-j-z-k}
|\alpha_j - z_k| \ge  (c_A^{-1} - c_A^{-2}) |\alpha_j|
\end{equation}
when $j \ge J(k)$. To see why \eqref{eq:alpha-j-z-k} holds true, observe that $\alpha_j \in D_{n(k)+\ell}$ for some $\ell \ge 2$, while $z_k \in D_{n(k)}$. 
Setting $\eta = \frac{4}{c_A-1}$, we have $\eta c_A^{n(k)+ \ell +2} \le |\alpha_j| \le \eta c_A^{n(k)+\ell+3}$ and $|z_k| \le \eta c_A^{n(k)+3}$. Hence
\begin{equation*}\begin{split}
|\alpha_j - z_k| &\ge \eta c_A^{n(k)+\ell+2} - \eta c_A^{n(k)+3} = \eta c_A^{n(k)+\ell +3} (c_A^{-1} - c_A^{-\ell}) \\
                       &\ge \eta c_A^{n(k)+\ell +3} (c_A^{-1} - c_A^{-2}) \\
                       &\ge |\alpha_j| \; (c_A^{-1} - c_A^{-2}).
\end{split}\end{equation*}
            
We now fix an index $k$ and consider several cases according to whether \eqref{case-1}, \eqref{case-2}, or \eqref{case-3} is valid.

\medskip

\paragraph{\textbf{Case I}} Assume that \eqref{case-1} is valid. Then $|z_k - z_{k+1}| = \delta$. By Lemma \ref{lemma:radial-growth} and the construction of $\Bar{X}$, $E(z_k) \ge e^{c_0 s_{n(k)}^{d-1}\delta}  \ge e^{c_0\delta}$ for $c_0 = \frac{1}{2}da_d$. Recalling the definition of $J(k)$ in \eqref{eq:Jk}, we observe that $|\alpha_j-z_k|\ge \varepsilon s_{n(k)}$ for all $j\le J(k)$. Hence
\[
\left| \frac{\alpha_j - z_{k+1}}{\alpha_j - z_k} \right| \ge 1 - \frac{|z_{k+1}-z_k|}{|z_k-\alpha_j|} \ge 1-\frac{\delta}{|\alpha_j-z_k|}\; \ge\; 1-\frac{\delta}{\varepsilon s_{n(k)}}\ =\ 1-\frac{\delta'}{c_A^{\,n(k)}},
\]
where $\delta' = \delta/\varepsilon$. Therefore
$$
Q_1(k) \ge \Bigl(1-\frac{\delta'}{c_A^{n(k)}}\Bigr)^{J(k)}.
$$         
where $Q_1(k)$ is as in \eqref{eq:Q-z-k}. Since $J(k)/c_A^{n(k)}\to0$ by~\eqref{eqn:J(k)_ratio_cn}, it follows that $Q_1(k)\longrightarrow 1$. Hence, for any $0<\nu\ll1$, there exists $N_1\in\N$ such that $Q_1(k) \ge 1-\nu$ for all $n(k) \ge N_1$. Define 
$$
K_1 := \inf\{k\in \N : z_k \in D_{\ge N_1}\}.
$$
Next, set $\lambda := c_A^{-1} - c_A^{-2}$. By \eqref{eq:alpha-j-z-k}, $|\alpha_j - z_k| \ge \lambda\,|\alpha_j|$ for all $j \ge J(k)$. Hence
$$
         \frac{\delta}{|\alpha_j- z_k|} \le \frac{\delta\lambda^{-1}}{|\alpha_j|}
$$
where $\delta>0$ is as in \eqref{eq:choice-of-delta}. Thus
$$
1 - \frac{|z_{k+1}-z_k|}{|\alpha_j-z_k|} \ge 1- \frac{\delta}{|\alpha_j- z_k|} \ge 1- \frac{\delta \lambda^{-1}}{|\alpha_j|}.
$$
Hence
$$
Q_2(k) \ge \prod_{j > J(k)+1} \left( 1- \frac{\delta \lambda^{-1}}{|\alpha_j|} \right)
$$
where $Q_2(k)$ is as in \eqref{eq:Q-z-k}. Since $\sum_j |\alpha_j|^{-1} < \infty$, the product on the right converges to one as $k\to\infty$. Thus, for every $0<\mu\ll1$, there exists $N_2 \in \N$ such that $Q_2(k) \ge 1-\mu$ for all $n(k) \ge N_2$. Define 
$$
K_2 := \inf\{k\in \N : z_k \in D_{\ge N_2}\}.
$$
We now choose small parameters $0<\mu,\nu\ll1$ and $\Lambda_I > 1$ so that $e^{c_0\delta} (1-\nu)(1-\mu) \ge \Lambda_I$. Then, letting 
$$
K_{I} := \max\{K_1, K_2\}\in \N
$$ 
we obtain for all $k\ge K_{I}$ that 
$$
\left| \frac{f(z_{k+1})}{f(z_{k})}\right| = E(z_k) Q(z_k) \ge \Lambda_I > 1.
$$ 
        
\paragraph{\textbf{Case II}} Now assume that either \eqref{case-2} or \eqref{case-3} is valid. Then $\varepsilon s_n \le |z_k - z_{k+1}| \le  (2c_A^2 + 2\varepsilon)s_n$. This case happens only when there exists $\alpha \in A$ so that $B(\alpha,\varepsilon)$ meets $[z_k,z_{k+1}]$. The latter fact implies that
$\varepsilon s_n \le |z_k - \alpha_j| \le (3c_A^2)s_n$ for all $\alpha_j \in  D_{\le n(k)+1}$. By Lemma~\ref{lemma:radial-growth},
        \begin{equation}\label{eqn:case2_e_z_k_bound}
            E(z_k) \ge e^{c_0 s_n^{d-1}(\varepsilon s_n)}  \ge e^{c_0\varepsilon s_n}.
        \end{equation}
Moreover, 
\begin{equation}\label{eqn:case2_Q1}
           Q_1(k)  = \prod_{j=1}^{J(k)+1}
            \Bigl|\frac{\alpha_j - z_{k+1}}{\alpha_j - z_k}\Bigr|
            \ge
            \prod_{j=1}^{J(k)+1}
            \frac{\varepsilon s_n}{3c_A^2s_n}
            = e^{-\vartheta (J(k)+1)},
       \end{equation}    
where $\vartheta := \log((3c_A^2)/(\varepsilon)) > 0$.

Next, using again \eqref{eq:alpha-j-z-k} we estimate
$$
\frac{|z_{k+1} -z_k|}{|\alpha_j- z_k|} \le \frac{|z_{k+1} -z_k|}{\lambda |\alpha_j|} \le \frac{(2c_A^2 + 2\varepsilon)s_n}{\lambda |\alpha_j|}
$$
and so
$$
\prod_{j > J(k)+1}^{\infty}1- \frac{|z_{k+1} -z_k|}{|\alpha_j- z_k|} \ge \prod_{j > J(k)+1}^{\infty}1- \frac{(2c_A^2 + 2\varepsilon)s_n}{\lambda |\alpha_j|} = \prod_{j > J(k)+1}^{\infty}1- \frac{\gamma s_n}{|\alpha_j|}  
$$
where $\gamma = \tfrac{2c_A^2+2\varepsilon}{\lambda}$. Observe that
$$
\frac{\gamma s_n}{|\alpha_j|} = \frac{(2c_A^2+2\varepsilon)s_n}{\lambda|\alpha_j|} = \frac{c_A^2 \, (2c_A^2+2\varepsilon)s_n}{(c_A-1)|\alpha_j|} \le \frac12
$$
since
$$
\frac{s_n}{|\alpha_j|} \le \frac{c_A-1}{4c_A^4}.
$$
The series $\sum_{j > J(k)+1}\log\Bigl(1-\frac{\gamma s_n}{|\alpha_j|}\Bigr)$ and $\sum_{j > J(k)+1} \frac{\gamma s_n}{|\alpha_j|}$ converge or diverge simultaneously, and we obtain
\begin{equation}\label{eqn:case2_Q2}
        \prod_{j > J(k)+1}^{\infty}1- \frac{|z_{k+1} -z_k|}{|\alpha_j- z_k|} \ge \exp \left( - 2\gamma s_n \sum_{j > J(k)+1} |\alpha_j|^{-1} \right).
\end{equation}
Combining \eqref{eqn:case2_e_z_k_bound}, \eqref{eqn:case2_Q1}, and \eqref{eqn:case2_Q2}, we conclude that
\begin{equation*}
            \left| \frac{f(z_{k+1})}{f(z_{k})}\right| \ge \exp \left( c_0 \varepsilon s_n - \vartheta (J(k) +1 ) - 2\gamma s_n \sum_{j > J(k)+1} |\alpha_j|^{-1} \right).
\end{equation*}
However, by Lemma \ref{lemma:lower_bound_annulus_dueto_alphaj}, 
$$
c_0 \varepsilon s_n - 2 \gamma s_n \sum_{j > J(k)+1} |\alpha_j|^{-1} \ge c_0 \frac{\varepsilon}{2}s_n,
$$
and hence 
\begin{equation*}
            \left| \frac{f(z_{k+1})}{f(z_{k})}\right| \ge \exp \left( c_0 \tfrac{\varepsilon}{2} s_{n(k)} - \vartheta J(k) \right).   
\end{equation*}
By \eqref{eqn:J(k)_ratio_cn},
\begin{equation*}
               \lim_{k\to \infty} \left| \frac{f(z_{k+1})}{f(z_{k})}\right| = \infty
\end{equation*}
Thus there exists $K_{II}$ and $\Lambda_{II} > 1$ such that $\left| \frac{f(z_{k+1})}{f(z_{k})}\right| \ge \Lambda_{II}$ for all $k > K_{II}$.

Finally, setting $K := \max \{ K_I, K_{II} \}$ and $\Lambda := \min \{ \Lambda_I, \Lambda_{II} \}$,
we conclude that $\left| \frac{f(z_{k+1})}{f(z_{k})}\right| \ge \Lambda > 1$ for all $k \ge K$. This completes the proof of Proposition \ref{prop:ratio_consecutive_terms_tildeX}. 

\section{Porosity and conformal mappings}\label{sec:porosity}

Following \cite{Vaisala1987Porous}, we say that a subset $E$ of a metric space $(X,d)$ is {\it porous in $X$} if there exists a constant $c>0$ so that for every $x \in X$ and every $r>0$ there exists $z \in B(x,r)$ with $B(z,cr) \cap E = \emptyset$. The connection between porosity and Nagata dimension is recalled in the following theorem.

\begin{theorem}\label{th:equivalence}
For a set $E \subset \R^n$, the following are equivalent:
\begin{itemize}
\item[(a)] $E$ is porous in $\R^n$,
\item[(b)] $\dim_A E < n$,
\item[(c)] $\dim_N E \le n-1$.
\end{itemize}
\end{theorem}

The equivalence of (a) and (b) was proved by Luukkainen \cite[Theorem 5.2]{luu:assouad}, while the implication (b) $\Rightarrow$ (c) follows from the fact that Nagata dimension is always bounded above by Assouad dimension \cite[Theorem 1.1]{DonneRajala2015}. Finally, the implication (c) $\Rightarrow$ (a) is due to Urs Lang; see \cite[Proposition 2.3]{SormaniWenger2010}.

This section is devoted to the proof of Theorem \ref{thm:spiral_domain_porosity_change}. The domain $\Omega$ will be constructed as a thickening of the polynomial spiral\footnote{The idea to use polynomial spiralling domains in this context was suggested to the authors by Sylvester Eriksson-Bique, and we are grateful for his permission to include the example here.}
\begin{equation}\label{eq:s-p}
S_p := \{ t^{-p} e^{\bi \pi t}:t \ge 1 \}, \quad p>0.
\end{equation}
Fraser \cite{fr:spiral} showed that the spirals $S_p$ have Assouad dimension equal to two for each $p>0$. By Theorem \ref{th:equivalence}, such sets also have Nagata dimension equal to two and are not porous in $\C$.

The argument which we will give is valid for any simply connected domain $\Omega \supset S_p$ with Jordan boundary so that the endpoints of $S_p$ lie in $\partial\Omega$. For simplicity, we work with the domain
\begin{equation}\label{eq:omega-p}
\Omega_p := \{ t^{-p} e^{\bi \pi t} e^{\bi \theta} : t > 1, |\theta|<\eps \}
\end{equation}
for some $0<\eps<\tfrac\pi2$. Note that $\overline{S_p} \subset \overline{\Omega_p}$ and that $\Omega_p$ is simply connected with Jordan boundary. See Figure \ref{fig:spiral-domain}.

\begin{figure}[h]
\includegraphics[width=3in]{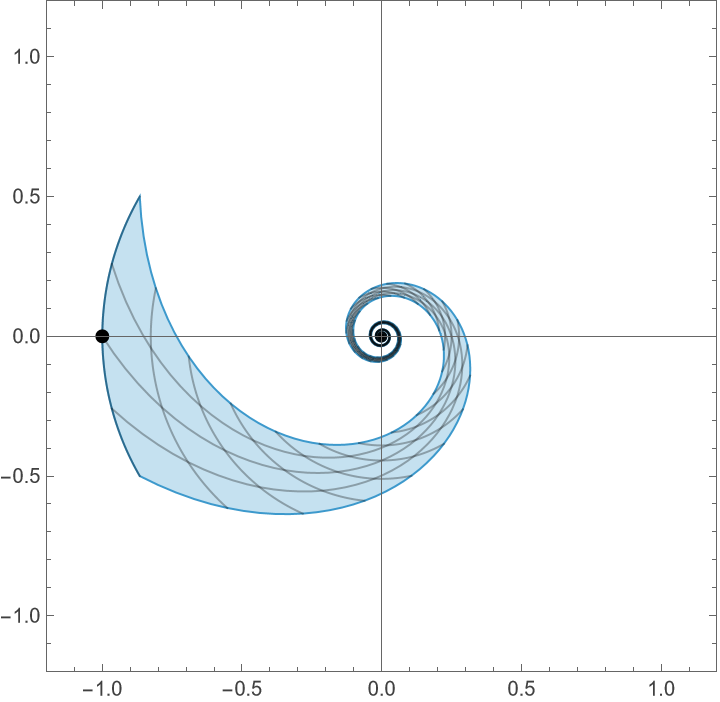}
\caption{The spiral domain $\Omega_p$}\label{fig:spiral-domain}
\end{figure}

\begin{proof}[Proof of Theorem \ref{thm:spiral_domain_porosity_change}]
Let $\Omega_p$ be as in \eqref{eq:omega-p} and let $f:\D \to \Omega_p$ be a Riemann map with continuous extension $f:\overline{\D} \to \overline{\Omega_p}$. By precomposing with an automorphism of $\D$ if necessary, we may ensure that $f(-1) = -1 \in \partial\Omega_p$ and $f(1) = 0 \in \partial\Omega_p$. Let $E = f([-1,1])$. The set $E$ is a closed and connected subset of $\overline{\Omega_p}$ with $-1,0 \in E$. Lemma \ref{lem:helper} below implies that $E$ is not porous in $\C$, which completes the proof.
\end{proof}

\begin{lemma}\label{lem:helper}
Let $E \subset \overline{\Omega_p}$ be a continuum with $-1,0 \in E$. Then $E$ is not a porous subset of $\C$.
\end{lemma}

\begin{proof}[Proof of Lemma \ref{lem:helper}]
Let $E$ be as in the statement of the lemma, and assume that $E$ is $c$-porous for some $c>0$. Without loss of generality we may assume that $c<\tfrac12$.

Choose $k_0 \in \N$ so that
\begin{equation}\label{eq:k0}
k_0 \ge \max \left\{ \frac{3p}{4c}, 2 \, \frac{(1-2c)^{1/p}}{1-(1-2c)^{1/p}} \right\}.
\end{equation}
Set $r_0 = (2k_0)^{-p}$. We will show that $B(0,r_0)$ contains no ball $B(z,r)$ with $r = cr_0$ and $B(z,r) \cap E = \emptyset$.

For $-\pi \le \alpha < \pi$, let $\Gamma_\alpha := \{ r e^{\bi\alpha}:r>0\}$ be the ray from $0$ with argument $\alpha$. Note that if a ball $B(z,r)$ as above exists, then there exists an interval $I$ of length $2cr_0$ along some ray $\Gamma_\alpha$ with $I \cap E = \emptyset$.

We now analyze the structure of
$$
\Omega_p \cap \Gamma_\alpha \cap B(0,r_0).
$$
Note that $z \in \Omega_p \cap \Gamma_\alpha \cap B(0,r_0)$ if and only if $z = t^{-p} e^{\bi \pi t} e^{\bi \theta}$ with $|\theta|<\eps$, $t>r_0^{-1/p} = 2k_0$, and $t \pi + \theta = \alpha + 2k \pi$ for some integer $k$. In this case, we have
$$
2k = t + \frac{\theta-\alpha}{\pi} > t - \frac{\eps+\pi}{\pi} > t - \frac32 \ge 2k_0 - \frac32 > 2k_0 - 2
$$
and so in fact $k \ge k_0$. Thus
$$
\Omega_p \cap \Gamma_\alpha \cap B(0,r_0) = \bigcup_{k=k_0+1}^\infty (a_k,b_k)e^{\bi\alpha} \cup \left( (a_{k_0},b_{k_0})e^{\bi\alpha} \cap B(0,r_0) \right),
$$
where
$$
a_k = (2k+ \frac{\alpha+\eps}{\pi})^{-p} \quad \mbox{and} \quad b_k = (2k + \frac{\alpha-\eps}{\pi})^{-p}.
$$
Hence
$$
\Gamma_\alpha \cap B(0,r_0) \setminus \Omega_p = \bigcup_{k=k_0+1}^\infty [b_{k+1},a_k] e^{\bi\alpha} \cup \left( ([b_{k_0+1},a_{k_0}] \cup [b_{k_0},a_{k_0-1}]) e^{\bi\alpha} \cap B(0,r_0) \right).
$$
The curve $E$ is a connected set joining $-1$ to $0$ in $\Omega_p$, and hence $E$ intersects both $(a_{k+1},b_{k+1})e^{\bi\alpha}$ and $(a_k,b_k)e^{\bi\alpha}$ for every $k \in \N$. It follows that the longest interval contained in $B(0,r_0) \cap \Gamma_\alpha \setminus E$ has length at most
$$
\max \left\{ \sup_{k \ge k_0+1} (b_k - a_{k+1}), r_0 - a_{k_0+1} \right\}.
$$
For $k \ge k_0+1$ we compute
$$
b_k - a_{k+1} = (2k+\frac{\alpha-\eps}\pi)^{-p} - (2k+2+\frac{\alpha+\eps}\pi)^{-p}.
$$
Using the identity $\int_A^B pw^{-p-1} \, dw = A^{-p} - B^{-p}$ for $A<B$, and the fact that $w \mapsto w^{-p-1}$ is decreasing, we obtain
\begin{equation*}\begin{split}
b_k - a_{k+1} &\le p (2k + \frac{\alpha-\eps}\pi)^{-p-1} (2k+2+\frac{\alpha+\eps}\pi - 2k - \frac{\alpha-\eps}\pi) \\
& = 2p(1+\frac\eps\pi)(2k+\frac{\alpha-\eps}\pi)^{-p-1} \\
&\le 3p (2k-\frac32)^{-p-1} \\
&\le 3p (2k_0+2-\frac32)^{-p-1}
< 2cr_0
\end{split}\end{equation*}
by \eqref{eq:k0}. Furthermore,
$$
(1-2c)r_0 = (1-2c)(2k_0)^{-p} < (2k_0+2+\frac{\alpha+\eps}\pi)^{-p}
$$
by \eqref{eq:k0}, and hence
$$
r_0 - a_{k_0+1} < 2cr_0
$$
as well.

Since $\alpha$ was arbitrary, it follows that $B(0,r_0) \setminus E$ can contain no radial line segment of length $2cr_0$. This contradicts the aforementioned consequence of the assumed $c$-porosity of $E$. We conclude that $E$ is not a porous subset of $\C$.
\end{proof}

\section{Further Discussion and Open Questions}\label{sec:open_questions}
    
Recall that a quasiregular mapping $f : \R^n \to \R^n$, $n \ge 2$, is said to be of {\it polynomial type} if $|f(x)| \to \infty$ as $|x| \to \infty$, whereas it is said to be of transcendental type if this limit does not exist. This is in direct analogy with the dichotomy between polynomials and transcendental entire functions. For more details, see \cite{FletcherNicks2011}.

A canonical example of a quasiregular mapping of polynomial type is the {\it power mapping} $f:\R^n \to \R^n$, $n \ge 3$; see e.g.\ \cite[Example I.3.2]{RickmanBookQRMaps} for the construction when $n=3$. This map preserves the Nagata dimension of arbitrary subsets of $\R^n$. The proof follows the same lines as our proof of Corollary \ref{cor:zn_preserves_ndim}, using the finite stability of Nagata dimension and the fact that conical sectors $\Omega(B) := \{ry:0<r<\infty, y \in B\}$ are uniform domains for sufficiently nice subsets $B \subset \Sph^{n-1}$.

\begin{question}
Let $f: \R^n \to \R^n$, $n \ge 3$, be a quasiregular map of polynomial type. Must $f$ preserve the Nagata dimension of arbitrary subsets of $\R^n$?
\end{question}

\begin{question}
Let $f:\Omega \to \R^n$, $\Omega \subset \R^n$, $n \ge 3$ be a quasiregular map. Must $\dim_N A = \dim_N f(A)$ for every $A \Subset \Omega$?
\end{question}

Our proofs of the corresponding results in dimension $n=2$ in section \ref{subsec:analytic_maps_ndim} rely heavily on the discreteness of the branch set (critical set) of an analytic function and the well-understood local behavior of the mapping near critical points. A significant complication in understanding the behavior of quasiregular mappings in higher dimensions lies in the more complicated structure of the branch set. Recall that the {\it branch set} $B_f$ of a quasiregular mapping $f:\Omega \to \R^n$ is
$$
B_f := \{ x \in \Omega : f \mbox{ is not locally homeomorphic at $x$} \}.
$$
In dimension at least three, the branch set of a quasiregular mapping is not discrete. By results by \v Cernavski\u i \cite{Cernavski1964} and V\"ais\"al\"a \cite{VaisalaBranchSet1966}, we have
\begin{equation}\label{eqn:cernavski}
\dim B_f = \dim f(B_f) \le n - 2.
\end{equation}
Here `$\dim$' denotes the topological dimension. It is not currently known (see, for example \cite[Question 29]{hs:questions} or \cite{sre:qr}) whether $\dim B_f$ is necessarily equal to $n-2$, when $f$ is a nonconstant quasiregular mapping in $\R^n$ with $B_f \ne \emptyset$.

It is natural to inquire whether a relationship analogous to~\eqref{eqn:cernavski} also holds for the Nagata dimension. However, the situation is quite different as demonstrated in the following examples.

\begin{example}
Let $f:\C \to \C$ be given by $f(z) = -\tfrac{1}{\pi} \cos(\pi z)$. Then $f$ is entire and hence $1$-quasiregular. The branch set $B_f = \Z$, while $f(B_f) = \{ \pm \tfrac1\pi \}$. Thus $\dim_N B_f = 1 > 0 = n-2$ and moreover $\dim_N B_f = 1 > \dim_N f(B_f) = 0$. 
\end{example}

\begin{example}\label{ex:zorich}
Let $f:\R^3 \to \R^3 \setminus \{0\}$ be Zorich's exponential-type quasiregular mapping. See e.g.\ \cite[Example I.3.3]{RickmanBookQRMaps}. For this map,
$$
B_f = \Z^2 \times \R
$$
while
$$
f(B_f) = \{(re^{\bi\theta},0):r>0,\theta \in \{0,\tfrac\pi2,\pi,\tfrac{3\pi}2\} \},
$$
whence $\dim_N B_f = 3 > \dim_N f(B_f) = 1$.
\end{example}
        
\begin{question}\label{q:nagata-dim-f-B-f}
Let $f:\Omega \to \R^n$, $n \ge 3$, be a quasiregular mapping. Must it hold true that $\dim_N f(B_f) \le n-2$?
\end{question}

\begin{question}
What are the possible values of $\dim_N B_f$ and $\dim_N f(B_f)$ for quasiregular mappings $f$ of domains in $\R^n$, $n \ge 3$?
\end{question}

Note that, in view of \eqref{eq:n-to-a}, Example \ref{ex:zorich} also provides an example of a quasiregular mapping $f:\R^n \to \R^n$ for which $B_f$ has Assouad dimension $n$. By way of contrast, it was shown in \cite[Theorem 5.13]{sar:qr-branch} and \cite[Theorem 1.3]{bh:qr-branch} that the Hausdorff dimensions of both $B_f$ and $f(B_f)$ are bounded above by a constant $c = c(n,K) < n$ whenever $f$ is a $K$-quasiregular mapping in $\R^n$, $n \ge 3$.

There exist quasiregular maps $f$ of $\R^n$, $n \ge 3$, for which $\dim_H B_f$ and $\dim_H f(B_f)$ take on any prescribed values in $[n-2,n)$. Gehring and V\"ais\"al\"a \cite[Corollary 22]{gv:hausdorff} provided such an example in connection with their estimates for quasiconformal distortion of Hausdorff dimension. In low dimensions, it can be deduced from a construction in \cite{bh:qr-branch} and \cite{ktw:qr-branch} that there exist quasiregular mappings whose branch sets exhibit a stronger property. To wit, when $n \in \{3,4\}$ and for any given $\alpha,\beta \in [n-2,n)$ there exists $f: \R^n \to \R^n$ quasiregular with $\dim_H B_f = \dim_A B_f = \alpha$ and $\dim_H f(B_f) = \dim_A f(B_f) = \beta$. In fact, both $B_f$ and $f(B_f)$ in all of the preceding examples are {\it quasiplanes}: there exist quasiconformal maps $g,h:\R^n \to \R^n$ so that $f = h \circ W_d \circ g^{-1}$ where $W_d:\R^n \to \R^n$ denotes the degree $d$ winding map \cite[Example I.3.1]{RickmanBookQRMaps}. Thus $B_f = g(\R^{n-2})$ and $f(B_f) = h(\R^{n-2})$, where $B_{W_d} = W_d(B_{W_d}) = \R^{n-2}$. Note that while the Hausdorff and Assouad dimensions of $f(B_f)$ are large in this class of examples, there is no contradiction with the conjecture in Question \ref{q:nagata-dim-f-B-f}. Indeed, both $B_f$ and $f(B_f)$ are quasisymmetrically equivalent to $\R^{n-2}$, and so the identities $\dim_N B_f = \dim_N f(B_f) = n-2$ follow from Theorem \ref{thm:QS_invariance_ndim}.

For an informative survey on the structure and complexity of the branch sets of higher-dimensional quasiregular mappings, we refer the reader to \cite{hei:icm}. 

\medskip

With regards to quasiconformal maps in higher dimensions, we anticipate that an example similar to that in Theorem \ref{thm:spiral_domain_porosity_change} can be constructed. In other words, we anticipate that the following question has a positive answer.

\begin{question}
Let $n \ge 3$. Does there exist a quasiconformal map $f:\Omega \to \Omega'$ be domains in $\R^n$, and a set $E \subset \Omega$ so that $\ndim E \le n-1$ and $\ndim f(E) = n$?
\end{question}

Higher-dimensional analogs of the polynomial spiralling domains $\Omega_p$ can be constructed, cf.\ the examples of {\it polynomial spiral shells} in \cite{ekc:spirals}, and we expect that it will be relatively straightforward to establish the non-porosity of appropriate subsets of such domains which wind through all levels of the spiral. A more challenging task is to show that such spiralling domains can be mapped quasiconformally onto uniform domains, or even onto the unit ball in $\R^n$.

\bibliographystyle{acm}
\bibliography{references}

@article {gv:hausdorff,
    AUTHOR = {Gehring, F. W. and V\"{a}is\"{a}l\"{a}, J.},
     TITLE = {Hausdorff dimension and quasiconformal mappings},
   JOURNAL = {J. London Math.\ Soc.\ (2)},
    VOLUME = {6},
      YEAR = {1973},
     PAGES = {504--512},
}

@article {bh:qr-branch,
    AUTHOR = {Bonk, Mario and Heinonen, Juha},
     TITLE = {Smooth quasiregular mappings with branching},
   JOURNAL = {Publ.\ Math.\ Inst.\ Hautes \'{E}tudes Sci.},
    NUMBER = {100},
      YEAR = {2004},
     PAGES = {153--170},
}

@article {sar:qr-branch,
    AUTHOR = {Sarvas, Jukka},
     TITLE = {The {H}ausdorff dimension of the branch set of a quasiregular
              mapping},
   JOURNAL = {Ann.\ Acad.\ Sci.\ Fenn.\ Ser.\ A I Math.},
    VOLUME = {1},
      YEAR = {1975},
    NUMBER = {2},
     PAGES = {297--307},
}

@article {fr:spiral,
    AUTHOR = {Fraser, Jonathan M.},
     TITLE = {On {H}\"{o}lder solutions to the spiral winding problem},
   JOURNAL = {Nonlinearity},
  FJOURNAL = {Nonlinearity},
    VOLUME = {34},
      YEAR = {2021},
    NUMBER = {5},
     PAGES = {3251--3270},
}

@article {ktw:qr-branch,
    AUTHOR = {Kaufman, Robert and Tyson, Jeremy T. and Wu, Jang-Mei},
     TITLE = {Smooth quasiregular maps with branching in {${\bf R}^n$}},
   JOURNAL = {Publ.\ Math.\ Inst.\ Hautes \'{E}tudes Sci.},
    NUMBER = {101},
      YEAR = {2005},
     PAGES = {209--241},
}

@book {heins:cft,
    AUTHOR = {Heins, Maurice},
     TITLE = {Complex function theory},
    SERIES = {Pure and Applied Mathematics, Vol.\ 28},
 PUBLISHER = {Academic Press, New York-London},
      YEAR = {1968},
     PAGES = {xv+416},
}

@article{LangSchlichenmaier2005,
    author = {Lang, Urs and Schlichenmaier, Thilo},
    title = "{Nagata dimension, quasisymmetric embeddings, and Lipschitz extensions}",
    journal = {Internat.\ Math.\ Research Notices},
    volume = {2005},
    number = {58},
    pages = {3625-3655},
    year = {2005},
    month = {01},
}

@article{DonneRajala2015,
 ISSN = {00222518, 19435258},
 URL = {http://www.jstor.org/stable/26315451},
 author = {Enrico {Le Donne} and Tapio Rajala},
 journal = {Indiana University Mathematics Journal},
 number = {1},
 pages = {21--54},
 publisher = {Indiana University Mathematics Department},
 title = {Assouad Dimension, {N}agata Dimension, and Uniformly Close Metric Tangents},
 volume = {64},
 year = {2015}
}

@article {luu:assouad,
    AUTHOR = {Luukkainen, Jouni},
     TITLE = {Assouad dimension: antifractal metrization, porous sets, and
              homogeneous measures},
   JOURNAL = {J. Korean Math.\ Soc.},
    VOLUME = {35},
      YEAR = {1998},
    NUMBER = {1},
     PAGES = {23--76},
}

@book {gh:quasidisk,
    AUTHOR = {Gehring, Frederick W. and Hag, Kari},
     TITLE = {The ubiquitous quasidisk},
    SERIES = {Mathematical Surveys and Monographs},
    VOLUME = {184},
 PUBLISHER = {American Mathematical Society, Providence, RI},
      YEAR = {2012},
     PAGES = {xii+171},
}

@article {ms:injectivity,
    AUTHOR = {Martio, O. and Sarvas, J.},
     TITLE = {Injectivity theorems in plane and space},
   JOURNAL = {Ann.\ Acad.\ Sci.\ Fenn.\ Ser.\ A I Math.},
    VOLUME = {4},
      YEAR = {1979},
    NUMBER = {2},
     PAGES = {383--401},
}

@article {BuyaloCap12005,
    AUTHOR = {Buyalo, S. V.},
     TITLE = {Capacity dimension and embedding of hyperbolic spaces into the
              product of trees},
   JOURNAL = {Algebra i Analiz},
    VOLUME = {17},
      YEAR = {2005},
    NUMBER = {4},
     PAGES = {42--58},
}

@article {BuyaloCap22005,
    AUTHOR = {Buyalo, S. V.},
     TITLE = {Asymptotic dimension of a hyperbolic space and the capacity
              dimension of its boundary at infinity},
   JOURNAL = {Algebra i Analiz},
    VOLUME = {17},
      YEAR = {2005},
    NUMBER = {2},
     PAGES = {70--95},
}

@article {BuyaloLebedevaPontryagin2007,
    AUTHOR = {Buyalo, S. V. and Lebedeva, N. D.},
     TITLE = {Dimensions of locally and asymptotically self-similar spaces},
   JOURNAL = {Algebra i Analiz},
    VOLUME = {19},
      YEAR = {2007},
    NUMBER = {1},
     PAGES = {60--92},
}

@article {BrodskiyDydakHigesMitra2007,
    AUTHOR = {Brodskiy, N. and Dydak, J. and Higes, J. and Mitra, A.},
     TITLE = {Dimension zero at all scales},
   JOURNAL = {Topology Appl.},
  FJOURNAL = {Topology and its Applications},
    VOLUME = {154},
      YEAR = {2007},
    NUMBER = {14},
     PAGES = {2729--2740},
      ISSN = {0166-8641,1879-3207},
   MRCLASS = {54F45 (18B30 54C55 54E35 54F50 54H11)},
  MRNUMBER = {2340955},
MRREVIEWER = {M.\ Hu\v{s}ek},
       DOI = {10.1016/j.topol.2007.05.006},
       URL = {https://doi.org/10.1016/j.topol.2007.05.006},
}

@article {Xie2008,
    AUTHOR = {Xie, Xiangdong},
     TITLE = {Nagata dimension and quasi-{M}\"{o}bius maps},
   JOURNAL = {Conform.\ Geom.\ Dyn.},
  FJOURNAL = {Conformal Geometry and Dynamics. An Electronic Journal of the
              American Mathematical Society},
    VOLUME = {12},
      YEAR = {2008},
     PAGES = {1--9},
      ISSN = {1088-4173},
   MRCLASS = {54F45 (30C65)},
  MRNUMBER = {2372759},
MRREVIEWER = {Matti\ Vuorinen},
       DOI = {10.1090/S1088-4173-08-00173-2},
       URL = {https://doi.org/10.1090/S1088-4173-08-00173-2},
}

@article {BrodskiyDydakLevinMitra2008,
    AUTHOR = {Brodskiy, N. and Dydak, J. and Levin, M. and Mitra, A.},
     TITLE = {A {H}urewicz theorem for the {A}ssouad-{N}agata dimension},
   JOURNAL = {J. Lond. Math. Soc. (2)},
  FJOURNAL = {Journal of the London Mathematical Society. Second Series},
    VOLUME = {77},
      YEAR = {2008},
    NUMBER = {3},
     PAGES = {741--756},
      ISSN = {0024-6107,1469-7750},
   MRCLASS = {54F45 (18B30 20F99 54E35)},
  MRNUMBER = {2418302},
MRREVIEWER = {Klaas\ Pieter\ Hart},
       DOI = {10.1112/jlms/jdn005},
       URL = {https://doi.org/10.1112/jlms/jdn005},
}

@article {DydakHiges2008,
    AUTHOR = {Dydak, J. and Higes, J.},
     TITLE = {Asymptotic cones and {A}ssouad-{N}agata dimension},
   JOURNAL = {Proc. Amer. Math. Soc.},
  FJOURNAL = {Proceedings of the American Mathematical Society},
    VOLUME = {136},
      YEAR = {2008},
    NUMBER = {6},
     PAGES = {2225--2233},
      ISSN = {0002-9939,1088-6826},
   MRCLASS = {54F45 (54C65 55M10)},
  MRNUMBER = {2383529},
       DOI = {10.1090/S0002-9939-08-09149-1},
       URL = {https://doi.org/10.1090/S0002-9939-08-09149-1},
}

@article {SormaniWenger2010,
    AUTHOR = {Sormani, Christina and Wenger, Stefan},
     TITLE = {Weak convergence of currents and cancellation},
   JOURNAL = {Calc. Var. Partial Differential Equations},
    VOLUME = {38},
      YEAR = {2010},
    NUMBER = {1-2},
     PAGES = {183--206},
}

@book {RemmertComplexFunctionTheory,
    AUTHOR = {Remmert, Reinhold},
     TITLE = {Classical topics in complex function theory},
    SERIES = {Graduate Texts in Mathematics},
    VOLUME = {172},
 PUBLISHER = {Springer-Verlag, New York},
      YEAR = {1998},
     PAGES = {xx+349},
       DOI = {10.1007/978-1-4757-2956-6},
       URL = {https://doi.org/10.1007/978-1-4757-2956-6},
}

@article {Vaisala1987Porous,
    AUTHOR = {V\"ais\"al\"a, Jussi},
     TITLE = {Porous sets and quasisymmetric maps},
   JOURNAL = {Trans.\ Amer.\ Math.\ Soc.},
    VOLUME = {299},
      YEAR = {1987},
    NUMBER = {2},
     PAGES = {525--533},
}

@article {hs:questions,
    AUTHOR = {Heinonen, Juha and Semmes, Stephen},
     TITLE = {Thirty-three yes or no questions about mappings, measures, and
              metrics},
   JOURNAL = {Conform.\ Geom.\ Dyn.},
    VOLUME = {1},
      YEAR = {1997},
     PAGES = {1--12 (electronic)},
}

@incollection {sre:qr,
    AUTHOR = {Srebro, Uri},
     TITLE = {Topological properties of quasiregular mappings},
 BOOKTITLE = {Quasiconformal space mappings},
    SERIES = {Lecture Notes in Math.},
    VOLUME = {1508},
     PAGES = {104--118},
 PUBLISHER = {Springer, Berlin},
      YEAR = {1992},
}

@inproceedings {hei:icm,
    AUTHOR = {Heinonen, Juha},
     TITLE = {The branch set of a quasiregular mapping},
 BOOKTITLE = {Proceedings of the {I}nternational {C}ongress of
              {M}athematicians, {V}ol.\ {II}},
     PAGES = {691--700},
 PUBLISHER = {Higher Ed. Press, Beijing},
      YEAR = {2002},
}

@book {Heinonen2001LecturesonAnalysis,
    AUTHOR = {Heinonen, Juha},
     TITLE = {Lectures on analysis on metric spaces},
    SERIES = {Universitext},
 PUBLISHER = {Springer-Verlag, New York},
      YEAR = {2001},
     PAGES = {x+140},
      ISBN = {0-387-95104-0},
   MRCLASS = {30C65 (28A75 28A78 46E35)},
  MRNUMBER = {1800917},
MRREVIEWER = {Christopher\ Bishop},
       DOI = {10.1007/978-1-4613-0131-8},
       URL = {https://doi.org/10.1007/978-1-4613-0131-8},
}

@book {RickmanBookQRMaps,
    AUTHOR = {Rickman, Seppo},
     TITLE = {Quasiregular mappings},
    SERIES = {Ergebnisse der Mathematik und ihrer Grenzgebiete (3)},
    VOLUME = {26},
 PUBLISHER = {Springer-Verlag, Berlin},
      YEAR = {1993},
     PAGES = {x+213},
      ISBN = {3-540-56648-1},
   MRCLASS = {30C65},
  MRNUMBER = {1238941},
MRREVIEWER = {Gaven\ J.\ Martin},
       DOI = {10.1007/978-3-642-78201-5},
       URL = {https://doi.org/10.1007/978-3-642-78201-5},
}

@article {FletcherNicks2011,
    AUTHOR = {Fletcher, Alastair N. and Nicks, Daniel A.},
     TITLE = {Quasiregular dynamics on the {$n$}-sphere},
   JOURNAL = {Ergodic Theory Dynam. Systems},
  FJOURNAL = {Ergodic Theory and Dynamical Systems},
    VOLUME = {31},
      YEAR = {2011},
    NUMBER = {1},
     PAGES = {23--31},
}

@article {Cernavski1964,
    AUTHOR = {{\v{C}}ernavski\u{i}, A. V.},
     TITLE = {Finite-to-one open mappings of manifolds},
   JOURNAL = {Mat.\ Sb.\ (N.S.)},
  FJOURNAL = {Matematicheski\u i\ Sbornik. Novaya Seriya},
    VOLUME = {65(107)},
      YEAR = {1964},
     PAGES = {357--369},
      ISSN = {0368-8666},
   MRCLASS = {54.60},
  MRNUMBER = {172256},
MRREVIEWER = {E.\ E.\ Floyd},
}

@article {VaisalaBranchSet1966,
    AUTHOR = {V\"ais\"al\"a, Jussi},
     TITLE = {Discrete open maps on manifolds},
   JOURNAL = {Ann.\ Acad.\ Sci.\ Fenn.\ Ser.\ A I},
    VOLUME = {392},
      YEAR = {1966},
     PAGES = {10},
}

@article {DavidNagataEmbed2024,
    AUTHOR = {David, Guy C.},
     TITLE = {A non-injective {A}ssouad-type theorem with sharp dimension},
   JOURNAL = {J. Geom. Anal.},
  FJOURNAL = {Journal of Geometric Analysis},
    VOLUME = {34},
      YEAR = {2024},
    NUMBER = {2},
     PAGES = {Paper No. 45, 19},
      ISSN = {1050-6926,1559-002X},
   MRCLASS = {30L99 (30L05 51F30)},
  MRNUMBER = {4682996},
MRREVIEWER = {Behnam\ Esmayli},
       DOI = {10.1007/s12220-023-01503-7},
       URL = {https://doi.org/10.1007/s12220-023-01503-7},
}

@article {ekc:spirals,
    AUTHOR = {Chrontsios-Garitsis, Efstathios-K.},
     TITLE = {On concentric fractal spheres and spiral shells},
   JOURNAL = {Nonlinearity},
  FJOURNAL = {Nonlinearity},
    VOLUME = {38},
      YEAR = {2025},
    NUMBER = {3},
     PAGES = {Paper No. 035019, 21},
}

@article {ChrontsiosTyson2023,
    AUTHOR = {Chrontsios-Garitsis, Efstathios K. and Tyson, Jeremy T.},
     TITLE = {Quasiconformal distortion of the {A}ssouad spectrum and
              classification of polynomial spirals},
   JOURNAL = {Bull. Lond. Math. Soc.},
  FJOURNAL = {Bulletin of the London Mathematical Society},
    VOLUME = {55},
      YEAR = {2023},
    NUMBER = {1},
     PAGES = {282--307},
      ISSN = {0024-6093,1469-2120},
   MRCLASS = {30C65 (28A78 30C62)},
  MRNUMBER = {4568342},
MRREVIEWER = {Lucas\ da Silva Oliveira},
       DOI = {10.1112/blms.12727},
       URL = {https://doi.org/10.1112/blms.12727},
}

@article {Chrontsios2024,
    AUTHOR = {Chrontsios-Garitsis, Efstathios K.},
     TITLE = {Quasiregular distortion of dimension},
   JOURNAL = {Conform.\ Geom.\ Dyn.},
    VOLUME = {28},
      YEAR = {2024},
     PAGES = {165--175},
}

@article {TukiaVaisala1980,
    AUTHOR = {Tukia, P. and V\"{a}is\"{a}l\"{a}, J.},
     TITLE = {Quasisymmetric embeddings of metric spaces},
   JOURNAL = {Ann.\ Acad.\ Sci.\ Fenn.\ Ser.\ A I Math.},
    VOLUME = {5},
      YEAR = {1980},
    NUMBER = {1},
     PAGES = {97--114},

}

@article {VaisalaQM1984,
    AUTHOR = {V\"ais\"al\"a, Jussi},
     TITLE = {Quasi-{M}\"obius maps},
   JOURNAL = {J. Analyse Math.},
  FJOURNAL = {Journal d'Analyse Math\'ematique},
    VOLUME = {44},
      YEAR = {1984/85},
     PAGES = {218--234},
}

@book {LehtoVirtanen1973,
    AUTHOR = {Lehto, Olli and Virtanen, K. I.},
     TITLE = {Quasiconformal mappings in the plane},
    SERIES = {Grundlehren der mathematischen Wissenschaften},
   EDITION = {2nd},
      VOLUME = {126},
 PUBLISHER = {Springer-Verlag, New York-Heidelberg},
      YEAR = {1973},
     PAGES = {viii+258},
}

@article {ass:plongements,
    AUTHOR = {Assouad, Patrice},
     TITLE = {Plongements lipschitziens dans {${\bf R}^{n}$}},
   JOURNAL = {Bull.\ Soc.\ Math.\ France},
  FJOURNAL = {Bulletin de la Soci\'{e}t\'{e} Math\'{e}matique de France},
    VOLUME = {111},
      YEAR = {1983},
    NUMBER = {4},
     PAGES = {429--448},

}

@article {AssouadOriginal,
    AUTHOR = {Assouad, Patrice},
     TITLE = {Sur la distance de {N}agata},
   JOURNAL = {C. R. Acad.\ Sci.\ Paris S\'er.\ I Math.},
  FJOURNAL = {Comptes Rendus des S\'eances de l'Acad\'emie des Sciences.
              S\'erie I. Math\'ematique},
    VOLUME = {294},
      YEAR = {1982},
    NUMBER = {1},
     PAGES = {31--34},
}

@article {tys:assouad,
    AUTHOR = {Tyson, Jeremy T.},
     TITLE = {Lowering the {A}ssouad dimension by quasisymmetric mappings},
   JOURNAL = {Illinois J. Math.},
  FJOURNAL = {Illinois Journal of Mathematics},
    VOLUME = {45},
      YEAR = {2001},
    NUMBER = {2},
     PAGES = {641--656},
}

@book {BuyaloSchroeder:asympbook,
    AUTHOR = {Buyalo, Sergei and Schroeder, Viktor},
     TITLE = {Elements of asymptotic geometry},
    SERIES = {EMS Monographs in Mathematics},
 PUBLISHER = {European Mathematical Society (EMS), Z\"urich},
      YEAR = {2007},
     PAGES = {xii+200},
      ISBN = {978-3-03719-036-4},
   MRCLASS = {53C23 (20F67)},
  MRNUMBER = {2327160},
MRREVIEWER = {Mario\ Bonk},
       DOI = {10.4171/036},
       URL = {https://doi.org/10.4171/036},
}
\end{document}